\documentclass[reqno]{amsart}
\newcommand \datum {\hfill {\small{\red{Version of February 4, 2026}}}}
\usepackage{amssymb,latexsym}
\usepackage{amsmath}
\usepackage{hyperref}
\usepackage{url}
\usepackage{graphicx}
\usepackage[dvipsnames]{xcolor}
\usepackage{enumerate}
\usepackage{enumitem}
\numberwithin{equation}{section}
\theoremstyle{plain}
 \newtheorem{theorem}{Theorem}[section]
 
 \newtheorem{lemma}[theorem]{Lemma}
 \newtheorem{corollary}[theorem]{Corollary}
 
 \newtheorem{observation}[theorem]{Observation}

\theoremstyle{definition}
 
 \newtheorem{convention}[theorem]{Convention}
\theoremstyle{remark}

\newcommand{\midvskip}{\vskip 5.5pt plus 1.5pt minus 1.5pt}

\providecommand{\logicbreak}{\par\vspace{0.5\baselineskip}}
\newcommand \nothing[1] {}

\newcommand \ul[1] {\underline{#1}}
\newcommand \ncu[1] {\textup{uc}{(#1)}}
\newcommand \ncd[1] {\textup{lc}{(#1)}}
\newcommand \jmref[1] {\textup{(jm\ref{#1})}}
\newcommand \eeqref[1] {\textup{(e\ref{#1})}}
\newcommand \ceqref[1] {\textup{(c\ref{#1})}}
\newcommand \peqref[1] {\textup{(p\ref{#1})}}
\newcommand \rgsiso{\cong_{\textup{rgs}}} 
\newcommand \dual{^{\textup{dual}}}
\newcommand \width[1]{\textup{width}(#1)} %
\newcommand \Bndl[1] {\textup{Bndl}(#1)}
\newcommand \Bndr[1] {\textup{Bndr}(#1)}
\newcommand \thejipair{\textup{pr}} %
\newcommand \jipair[1]{\thejipair(#1)} %
\newcommand \ajipair[2]{\thejipair_{#1}(#2)} %
\renewcommand \phi{\varphi}
\newcommand \theFoot{\textup{Foot}}
\newcommand \Foot [2]{\theFoot(#1,#2)}
\newcommand \SFoot [3]{\theFoot_{#1}(#2,#3)}
\newcommand \Nplu {{\mathbb N^+}}
\newcommand \Nnul {{\mathbb N_0}}
\newcommand{\tbf}{\textbf}
\newcommand{\set}[1]{\{#1\}}
\newcommand\figwidthcoeff{0.95}
\newcommand \lcd{\textup{lcd}} 
\newcommand \cd[1]{\textup{cd}(#1)}
\newcommand \lnc [2] {\textup{lnc}(#1,#2)} 
\newcommand \Con [1] {\textup{Con}(#1)}  
\newcommand \Jir [1] {\textup{Ji}(#1)}
\newcommand \Mir [1] {\textup{Mi}(#1)}
\newcommand \Jr [1] {\textup{Jr}(#1)}
\newcommand \Mr [1] {\textup{Mr}(#1)}
\newcommand \chain [1] {\mathsf{C}_{#1}}
\newcommand \Nar [1] {\textup{Nar}(#1)}
\newcommand \Core [1] {\textup{Cor}(#1)}
\newcommand \gsum {\mathrel{\dotplus}}
\newcommand \Circ [1] {\textup{Circ}(#1)}
\newcommand \idl [1] {\textup{idl}(#1)}
\newcommand \fil [1] {\textup{fil}(#1)}
\newcommand \idla [2] {\textup{idl}_{#1}(#2)}
\newcommand \fila [2] {\textup{fil}_{#1}(#2)}
\newcommand \Idl [1] {\textup{Idl}(#1)}
\newcommand \con {\textup{con}}
\newcommand \Edge [1] {\textup{Edge}(#1)}
\newcommand \pup {\nearrow_\textup{p}}
\newcommand \pdn {\searrow_\textup{p}}
\newcommand \pers {\sim_{\textup{p}}}
\newcommand \defiff {\overset{\scriptscriptstyle{\textup{def}}}\iff}
\newcommand \locov[1] {{#1}_\ast} 
\newcommand \rleq{\leq_\rho}
\newcommand \req{\equiv_\rho}
\newcommand \quot[1] {``#1''}
\renewcommand \epsilon{\varepsilon}
\newcommand \doi[1] {\href{https://doi.org/#1}{DOI: #1}} 

\newcommand \red[1]{{\textcolor{red}{#1}\color{black}}}

\begin{document}

\title[Lattices with congruence densities larger than $3/32$]
{\boldmath Lattices with congruence densities larger than $3/32$}


\author[G.\ Cz\'edli]{G\'abor Cz\'edli}
\email{czedli@math.u-szeged.hu}
\urladdr{http://www.math.u-szeged.hu/~czedli/}
\address{University of Szeged, Bolyai Institute. 
Szeged, Aradi v\'ertan\'uk tere 1, HUNGARY 6720}

\begin{abstract} 
By a 1997 result of R.\ Freese, an $n$-element lattice has at most $2^{n-1}$ congruences. This motivates us to define the \emph{congruence density} $\cd L$  of a finite $n$-element lattice as $|\Con L|/2^{n-1}$, where $|\Con L|$ is the number of elements of the congruence lattice $\Con L$ of $L$. We prove that whenever $L$ is a finite lattice with $\cd L>3/32$, then $L$ has the same number of join-irreducible and meet-irreducible elements. This result is sharp, since there exists a six-element lattice $R_6$ with $\cd {R_6}=3/32$ but fewer join-irreducible than meet-irreducible elements.
By R.\ Freese, C.\ Mure\c{s}an, J.\ Kulin, and the present author's results, lattices with congruence densities  
larger than $1/8$ have already been described. Here we decrease the lower threshold from $1/8$ to $3/32$. That is, we describe all finite lattices $L$ such that $\cd L>3/32$. As a corollary, we give the $k$th largest number of congruences of $n$-element lattices for 
$n>8$ and $k\in\set{n+1, n+2, n+3,n+4}$.
\end{abstract}

\dedicatory{This paper is dedicated to Eszter K.\ Horv\'ath, my coauthor on ten papers.}

\thanks{This research was supported by the National Research, Development and Innovation Fund of Hungary, under funding scheme K 138892.  
\hfill{\red{\tbf{\datum}}}}



\maketitle

\section{Introduction and the first theorem}\label{sect:intro}
We assume that the reader is familiar with the basic concepts of lattice theory or universal algebra.
Several auxiliary results from  Cz\'edli \cite{czg864} (open access at the time of writing) 
and from other sources will be used, but all of them will be recalled when needed.  

All lattices considered in this paper are assumed to be \emph{finite} and $L$ always denotes a finite \emph{lattice}, even when these conventions are not explicitly repeated.
For a lattice $L=(L;\vee,\wedge)$, we denote by $\Con L$ the lattice of congruence relations of $L$, in short, the \emph{congruence lattice} of $L$. Then the number of congruences of $L$ is $|\Con L|$. As the main concept in this paper, we define the \emph{congruence density} $\cd L$ of a finite lattice $L$ by $\cd L:=|\Con L|/2^{|L|-1}$. 
Since $|\Con L|\leq 2^{|L|-1}$ by a result of Freese \cite{freese}, $\cd L\leq 1$. 
Note that $\cd L=1$ if and only if $L$ is a chain.
Our first goal is to prove the following theorem.

\begin{theorem}\label{thm:jm} Every finite lattice $L$ with $\cd L>3/32$ has the following two properties.
\begin{enumerate}
\renewcommand{\labelenumi}{\textup{(jm\theenumi)}}
  \item\label{jm1} 
 $L$ has the same number of join-irreducible and meet-irreducible elements.
  \item\label{jm2} 
 More generally, for every nonnegative integer $k$, the number of elements of $L$ covering exactly $k$  elements equals the number of elements covered by exactly $k$ elements.
\end{enumerate}
\end{theorem}

The letters in \quot{jm} come from \quot{\ul join-irreducible} and  \quot{\ul meet-irreducible}.
As indicated by the phrase \quot{more generally}, property \jmref{jm2} straightforwardly implies \jmref{jm1}. 
By Dilworth's Covering Theorem \cite{Dilworth} (see also Kung \cite[Theorem 464]{Kung}), every finite modular lattice has property \jmref{jm2} and thus also property \jmref{jm1}. 
Note that the modularity of $L$ and the condition $\cd L>3/32$ are independent, that is, neither implies the other. 

As the following observation (to be proved later) shows, Theorem \ref{thm:jm} is sharp.

\begin{observation}\label{obs:R6} The six-element lattice $R_6$ in the middle of Figure \ref{figalma} has only three join-irreducible but four meet-irreducible elements, and $\cd{R_6}=3/32$.
\end{observation}

Still in connection with Theorem \ref{thm:jm}, let us point out that $\cd L$ has a similar effect on another lattice property: Cz\'edli \cite{czg-manconplan} proved that for every finite lattice $L$, 
\begin{equation}
\text{if }\cd L>1/16,\text{ then }L\text{ is planar}.
\label{eq:233planar}
\end{equation}
The eight-element Boolean lattice, which is not planar and has congruence density $1/16$, witnesses that this result is also sharp.

Results proved by Freese \cite{freese}, Mure\c{s}an and Kulin \cite{muresankulin}, and Cz\'edli \cite{czg-lconl2} and \cite{czg864} 
present a reasonable structural description of all finite lattices $L$ such that  $\cd L>1/8$; this description is obtained by combining Lemmas 3 and 10 of Cz\'edli \cite{czg864}, and it is presented here as Lemma \ref{lemma:earlylist}. 
Motivated by this result and Theorem \ref{thm:jm}, our main goal is to improve the threshold from $1/8$ to $3/32$.  

Although the passage from $1/8=4/32$ to $3/32$ may appear minor, it requires nontrivial work. In particular, while the lattices with $\cd L>1/8$ are of width at most 2 (i.e., without 3-element antichains), this no longer holds for lattices with $\cd L>3/32$. 
The passage 
in question also requires two new methods. 
First,  in this paper, we occasionally work with \quot{distant} elements of $\Foot a b $ (to be defined in due course); see, for example, Lemma \ref{lemma:diffeet} and the proofs of Lemmas \ref{lemma:khVnMtTlkBr} and \ref{lemma:3lowcov}. 
Second, we can now apply the result of Cz\'edli \cite{czgcdofS}\footnote{\cite{czgcdofS} was written and submitted later than \cite{czg864}.}, formulated here in \eqref{eq:dsmntlExtsm}, to certain subsets; for instance, see the proofs of Lemmas \ref{lemma:circedgglu} and \ref{lemma:w2}.

\section{Stating the Main Theorem and its corollary}\label{sect:mainthm}
We need to borrow some notations and notions from the folklore and Cz\'edli \cite{czg864}. 
For $4\leq n\in\Nplu:=\set{1,2,3,\dots}$, let $\Circ n$ denote the class of $n$-element lattices whose (Hasse) diagrams are $n$-element circles, that is, cycles of length $n$ in graph-theoretic terms. For example, $\Circ 6$ consists of two lattices\footnote{\label{foot:nodistinction}Here and later, we do not distinguish between two isomorphic lattices.}, $N_6$ and $N_6'$,  depicted in Figure \ref{figalma}.


\begin{figure}[ht] 
\centerline{ \includegraphics[width=\figwidthcoeff\textwidth]{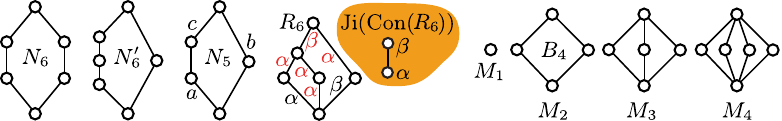}} 
\caption{Some lattices occurring in the paper}\label{figalma}
\end{figure}

For finite lattices $L_1$ and $L_2$, the \emph{glued sum} $L_1\gsum L_2$ is the special Hall-Dilworth gluing of $L_1$ and $L_2$ such that the intersection of $L_1$ and $L_2$ (which are sublattices of $L_1\gsum L_2$) is a one-element sublattice. In other words, we obtain the (Hasse) diagram of $L_1\gsum L_2$ by putting $L_2$ atop $L_1$, and identifying the top $1_{L_1}$ of $L_1$ with the bottom $0_{L_2}$ of $L_2$. Note that $(L_1\gsum L_2)\gsum L_3=L_1\gsum(L_2\gsum L_3)$, so we can drop the parentheses and we can write expressions like $L_1\gsum \dots \gsum L_h$.
Another special Hall-Dilworth gluing of $L_1$ and $L_2$ is their \emph{edge gluing}, where $|L_1\cap L_2|=2$. 
\nothing{Sometimes we will refer to this construction as the \emph{edge gluing operation}, but keep in mind that it is not single-valued, since it depends not only on the summands but also on the choice of the edges to be glued.}
For $n\in\Nplu:=\set{1,2,3,\dots}$, $\chain n$ will stand for the $n$-element \emph{chain}. A subset $Y$ of a lattice $L$ is \emph{convex} if for all $x,z\in Y$ and $y\in L$, $x\leq y\leq z$ implies $y\in Y$. 
With the notation 
\begin{equation*}
\Nar L:=\{x\in L: x\text{  is comparable with every element of }L\},
\end{equation*} 
a non-singleton set $X$ is a \emph{narrow chain component} of $L$ if it is a maximal subset of $\Nar L$ that is a convex subset of $L$.  
Let $L$ be a non-singleton lattice. Then there is a unique $h\in \Nplu$ and there are uniquely determined elements $u_0=0_L<u_1<\dots <u_h=1_L$ such that for each $i\in\set{1,\dots,h}$,  the \emph{interval} $[u_{i-1}, u_i]:=\{x\in L: u_{i-1} \leq x\leq u_i\}$ is either a narrow chain component of $L$ or $[u_{i-1},u_i]\cap\Nar L=\set{u_{i-1}, u_i}\neq [u_{i-1},u_i]$.  In this case, 
\begin{equation}
L=[u_0,u_1]\gsum[u_1,u_2]\gsum\dots\gsum[u_{h-1},u_h],  
\label{eq:cangsumdecomp}
\end{equation}
which is the \emph{canonical glued sum decomposition} of $L$.
This is exemplified by $L$, $K$ (with $h=3$), and $M$ (with $h=4$) in Figure \ref{figszilva}. If $L$ is a non-singleton chain, then $h=1$,  \eqref{eq:cangsumdecomp} turns into $L=[u_0,u_1]=[0_L,1_L]$, and $L$ is the only narrow chain component of itself. 
The one-element lattice has neither a canonical glued sum decomposition nor a narrow chain component.


\begin{figure}[ht] 
\centerline{ \includegraphics[width=\figwidthcoeff\textwidth]{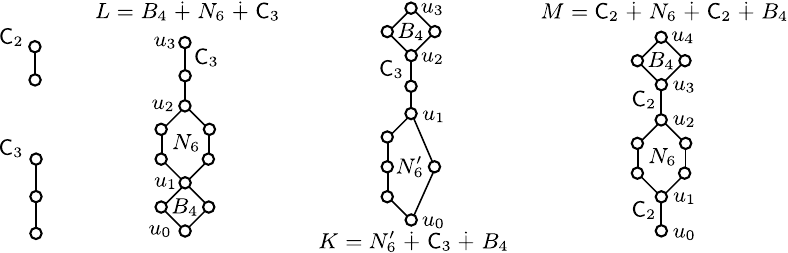}} 
\caption{Illustrations for the canonical glued sum decompositions}\label{figszilva}
\end{figure}

We say that a finite lattice $L$ is a \emph{core lattice} if $x\prec_L y$ holds for no $x,y\in\Nar L$ or, equivalently, none of the summands in the canonical glued sum decomposition \eqref{eq:cangsumdecomp} of  $L$ is a chain. Note that the one-element lattice is a core lattice.
For core lattices $L$ and $K$, we say that 
$L$ and $K$ are \emph{rearranged-glued-sum isomorphic} (in notation, $L\rgsiso K$) if, using the notation \eqref{eq:cangsumdecomp} for the canonical glued sum decomposition of $L$, there is a permutation $\pi$ of $\set{1,\dots,h}$ such that 
\begin{equation}
K = [u_{\pi(1)-1},u_{\pi(1)}]\gsum[u_{\pi(2)-1},u_{\pi(2)}]\gsum\dots \gsum[u_{\pi(h)-1},u_{\pi(h)}].
\label{eq:rgscisomor}
\end{equation}
Let us emphasize that $L\rgsiso K$ is defined only if $L$ and $K$ are core lattices.

If $L$ is a finite chain, then its \emph{core}, denoted by $\Core L$, is the one-element lattice. Assume that $L$ is a finite lattice that is not a chain. With respect to \eqref{eq:cangsumdecomp}, 
let $L_i:=[u_{i-1},u_i]$ for $i\in\set{1,\dots,h}$, and  
let $J:=\{i: 1\leq i\leq h$ and $L_i$ is not a chain$\}$. Then $J$ is nonempty and has the form $J=\set{j_1,\dots, j_t}$ such that $1\leq j_1<j_2<\dots <j_t\leq h$, and we define the \emph{core} of $L$ by
\begin{equation}
\Core L :=L_{j_1}\gsum L_{j_2}\gsum \cdots \gsum L_{j_t}.
\label{eq:CoreLdotsum}
\end{equation}
Note that, regardless of whether $L$ is a chain or not, 
$\Core{\Core L}=\Core L$, and $\Core L$ is a core lattice.
Cores are convenient tools for describing most of the classes of finite lattices we need. 
(Here, based on Footnote \ref{foot:nodistinction}, we could say \quot{sets} instead of \quot{classes}.)
For example, 
\begin{equation*}
\{L: |L|=11 \text{ and }\Core L\rgsiso B_4\gsum X\text{ for some }X\in \Circ 6 \}, 
\end{equation*}
where $B_4$ stands for the four-element Boolean lattice drawn in Figure \ref{figalma}, is a concise description of a 24-element set of lattices, and $L$, $K$, and $M$ shown in Figure \ref{figszilva} belong to this set.  Similarly, $\{L: \Core L \rgsiso B_4\gsum X$ for some $X\in \Circ 6\}$ describes an infinite set of finite lattices; since $|L|\geq |\Core L|$, each lattice in this class (or set) has at least 9 elements. 
Quite often,  $\Core L$ will be interesting for us only up to rearranged-glued-sum isomorphism. However, our rearranged-glued-sum isomorphism reduces to the usual isomorphism when $L_{j_1} \cong L_{j_2} \cong\dots\cong L_{j_t}$ in  \eqref{eq:CoreLdotsum} (in particular, if $t=1$).

As in \cite{czg864}, we will often express congruence densities in the form $x/64$, even when  simplification would be possible or the numerator is not an integer. The common denominator 64 facilitates comparison and improves readability. 
We recall the following lemma from  Cz\'edli \cite[Lemmas 3 and 10]{czg864}, but note that its 
\eeqref{lem1}, \eeqref{lem2}, and \eeqref{lem3}--\eeqref{lem5} parts  were proved in  Freese \cite{freese}, Cz\'edli \cite{czg-lconl2}, and  Mure\c san and Kulin \cite{muresankulin}, respectively. The letter e in the labels \eeqref{lem1}--\eeqref{lem10} indicates that we display results from \emph{earlier} papers. Similarly, the letters c and p in the labels \ceqref{th1}--\ceqref{th8} and \peqref{lemma:1nmpTn}--\peqref{lemma:8nmpTn} will refer to the \emph{current} paper and to \emph{previous} results, respectively.

\begin{lemma}[Cz\'edli {\cite[Lemmas 3 and 10]{czg864}}]\label{lemma:earlylist} 
Let $L$ be a finite lattice. Then $\cd L>8/64$ if and only if one of the following ten conditions holds.
\begin{enumerate}
\renewcommand{\labelenumi}{\textup{(e\theenumi)}}
\item\label{lem1} $|\Core L|=1$, that is, $L$ is a chain. (In this case, $\cd L=64/64$.)
\item\label{lem2} $\Core L \cong B_4$. (In this case, $\cd L=32/64$.)
\item\label{lem3} $\Core L \cong N_5$. (In this case, $\cd L=20/64$.)
\item\label{lem4} $\Core L \cong \chain 2\times \chain 3$. Equivalently, $\Core L$ is an edge gluing of two copies of $B_4$. (In this case, $\cd L=16/64$.)
\item\label{lem5} $\Core L \cong B_4\gsum B_4$. (In this case, $\cd L=16/64$.)
\item\label{lem6} $\Core L \in \Circ 6$. (In this case, $\cd L=14/64$.)
\item\label{lem7} $\Core L \in \Circ 7$. (In this case, $\cd L=11/64$.)
\item\label{lem8} $\Core L$ is an edge gluing of $B_4$ and $N_5$, in either order. (In this case, $\cd L=10/64$.)
\item\label{lem9} $\Core L\rgsiso B_4\gsum N_5$. (In this case, $\cd L=10/64$.)  
\item\label{lem10} For some $8\leq n\in\Nplu$, $\Core L\in\Circ n$. (In this case,  
we have that
$\cd L=(8+3/2^{n-7})/64$.)
\end{enumerate}
\end{lemma}

\begin{figure}[ht] 
\centerline{ \includegraphics[width=\figwidthcoeff\textwidth]{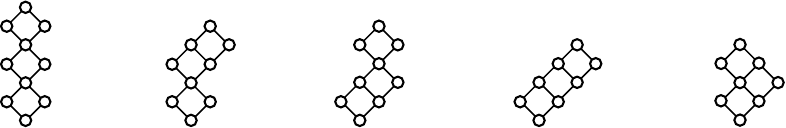}} 
\caption{The core lattices in Theorem \ref{thm:main}\ceqref{th3} }\label{figcitrom}
\end{figure}

We are now ready to present the main result of the paper. The lattices $B_4$, $M_3$, $N_5$, $N_6$, and $N_6'$ appearing in the theorem below are defined by Figure \ref{figalma}, while the symbol $\rgsiso$ was  introduced earlier, in the sentence including \eqref{eq:rgscisomor}. Note that each occurrence of $\cong$ in the theorem could be replaced by $\rgsiso$, but not conversely.

\begin{theorem}[Main Theorem]\label{thm:main}
For a finite lattice $L$, the congruence density of $L$ exceeds $3/32$ (i.e., $\cd L>6/64$) if and only if one of the following eight conditions is satisfied.
\begin{enumerate}
\renewcommand{\labelenumi}{\textup{(c\theenumi)}}
  \item\label{th1} 
$\cd L>8/64$; for a complete characterization of lattices with this property, see Lemmas 3 and 10 in \cite{czg864} or, equivalently, see Lemma \ref{lemma:earlylist} here. 
  \item\label{th2} 
     $\Core L \cong M_3$.  In this case,  $\cd L=8/64$.                      
  \item\label{th3} 
     $\Core L\cong B_4 +_1 B_4 +_2 B_4$, where each of the symbols $+_1$ and $+_2$ denotes either the glued sum operation $\gsum$ or an edge gluing. In other words, $\Core L$ is isomorphic to one of the five lattices given in Figure~\ref{figcitrom}.  In this case, $\cd L=8/64$.
  \item\label{th4} 
$\Core L\cong N_{5,5}$; see the left side of Figure \ref{figdio}.  In this case, $\cd L=7/64$.
  \item\label{th5} 
For some $X\in \Circ 6$, $\Core L \rgsiso B_4 \gsum X$. In other words, $\Core L \rgsiso B_4 \gsum N_6$ or $\Core L \rgsiso B_4 \gsum N'_6$. In this case, $\cd L=7/64$.
  \item\label{th6} 
$\Core L$ is obtained by an edge gluing of $B_4$ and an $X\in\Circ 6$ in one of the two possible orders. That is, $\Core L$ is isomorphic to one of the six lattices drawn in Figure \ref{figeper}. In this case, $\cd L=7/64$.
  \item\label{th7} 
$\Core L$ is an edge gluing of two copies of $N_5$. In this case,  $\cd L=6.5/64$.
  \item\label{th8} 
$\Core L \cong N_5\gsum N_5$. In this case, $\cd L=6.25/64$. 
\end{enumerate}
\end{theorem}


\begin{figure}[ht] 
\centerline{ \includegraphics[width=\figwidthcoeff\textwidth]{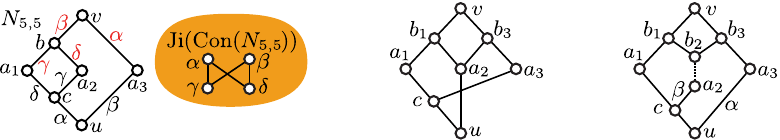}} 
\caption{$N_{5,5}$ and illustrations for the proof of Lemma \ref{lemma-3a-cha}}\label{figdio}
\end{figure}


\begin{figure}[ht] 
\centerline{ \includegraphics[width=\figwidthcoeff\textwidth]{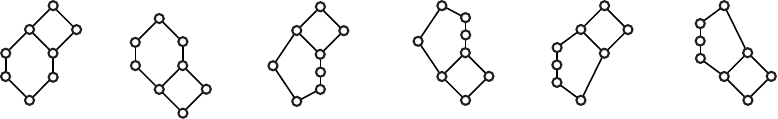}} 
\caption{The core lattices in Theorem \ref{thm:main}\ceqref{th6} }\label{figeper}
\end{figure}

For $n,k\in\Nplu$, let $\lcd (n, k)$ denote the $k$-th \ul largest value among the \ul congruence \ul densities of $n$-element lattices. Note that $\lcd(n,k)$ is not always defined. For example, $\lnc 3 2$ and $\lnc 4 3$ are undefined, since $|\{\cd L: |L|=3\}|=1$ and $|\{\cd L: |L|=4\}|=2$. 
With straightforward reformulations of parts \peqref{lemma:1nmpTn}--\peqref{lemma:5nmpTn},   
we recall the following statement from four different papers.

\begin{lemma}\label{lemma:nmpTn} \
\begin{enumerate}
\renewcommand{\labelenumi}{\textup{(p\theenumi)}}
\item\label{lemma:1nmpTn} 
  $\lcd(n,1)=1=64/64$ for all $n\in\Nplu$  (proved by Freese \cite{freese}).
\item $\lcd(n,2)=32/64$ for $4\leq n\in\Nplu$ (proved by Cz\'edli \cite{czg-lconl2}).
\item $\lcd(n,3)=20/64$ for $5\leq n\in\Nplu$ (proved by Mure\c{s}an and Kulin \cite{muresankulin}).
\item $\lcd(n,4)=16/64$ for $6\leq n\in\Nplu$ (proved by Mure\c{s}an and Kulin \cite{muresankulin}). 
\item\label{lemma:5nmpTn} 
 $\lcd(n,5)=14/64$ for $6\leq n\in\Nplu$  (proved by Mure\c{s}an and Kulin \cite{muresankulin}). 
\item $\lcd(n,6)=11/64$ for $7\leq n\in\Nplu$ (proved by Cz\'edli \cite{czg864}).
\item  $\lcd(n,7)=10/64$ for $7\leq n\in\Nplu$  (proved by Cz\'edli \cite{czg864}).
\item\label{lemma:8nmpTn} 
$\lcd(n,k)=(8+3\cdot 2^{7-k})/64$ for $8\leq k\leq n$ (proved by Cz\'edli \cite{czg864}).
\end{enumerate}
\end{lemma}

In particular, if $8\leq n\in\Nplu$, then 
$\lcd(n,k)$ is known and is given by a simple formula for each $k\in\set{1,2,\dots,n}$. 
By tedious work, Section 4 of Cz\'edli \cite{czg864} presents $\lcd(7,k)$ for all $k$ such that $\lcd(7,k)$ is defined.
In the following statement, we establish three new values of $\lcd(n,k)$ for $n = 8$ and four\footnote{\label{foot:imposs}Obtaining a few further values for small $n$ 
with computer assistance would be possible, but one cannot go too far in this direction.  
Determining a few further values by theoretical considerations over dozens of pages might also be possible in the future.  However, computing,  for example,  $\lcd(10^{10}, 10^{12})$ in the  foreseeable future seems unlikely.} 
new values for $n>8$.

\begin{corollary}\label{corol:1}\ 

\textup{(A)} For $9\leq n\in\Nplu$, we have that $\lcd(n,n+1)=8/64$, $\lcd(n,n+2)=7/64$, $\lcd(n,n+3)=6.5/64$, and $\lcd(n,n+4)=6.25/64$. Furthermore,  $\lcd(8,9)=8/64$, $\lcd(8,10)=7/64$, and  $\lcd(8,11)=6.5/64$. 

\textup{(B)} Multiplying the $\lcd(n,k)$ values given in part \textup{(A)} by $2^{n-1}$, we obtain the 
\emph{$k$th} largest number of congruences of $n$-element lattices.
\end{corollary}

The lattices witnessing the $\lcd$-values given above are described in Theorem \ref{thm:main}. 
Corollary \ref{corol:1} follows immediately from
Theorem \ref{thm:main} and Lemmas \ref{lemma:earlylist} and \ref{lemma:nmpTn}; thus no separate proof will be given.

\section{Tools}
The congruence density of a lattice depends only on its core. Indeed, recall from Cz\'edli \cite[Lemma 3]{czg864} that for every finite lattice $L$, 
\begin{equation}
\cd L=\cd{\Core L}.
\label{eq:cdLcdCorL}
\end{equation}
A \emph{quasiorder} on a set is a reflexive and transitive relation.
If $\rho$ is such, then $a\rleq b$ and $a\req b$ will stand for $(a,b)\in\rho$ and $(a,b)\in\rho\cap\rho^{-1}$, respectively.
For a \emph{quasiordered set} $(A;\rho)$, a subset $B$ of $A$ is an \emph{ideal} if for all $x,y\in A$,  $y\in B$ and $x\rleq y$ imply $x\in B$. The ideals of $(A,\rho)$ form a distributive lattice $\Idl{A;\rho}=(\Idl{A;\rho};\cup,\cap)$. When the quasiorder $\rho$ above is antisymmetric, $(A,\rho)$ is a \emph{poset}.

For a finite lattice $L$, let $\Jir L$ stand for the \emph{set of join-irreducible elements} of $L$; an element of $L$ is \emph{join-irreducible} if it has exactly one lower cover. The (dually defined) \emph{set of} \emph{meet-irreducible elements} is denoted by $\Mir L$.
The unique lower cover of an element $x\in\Jir L$ will be denoted by $\locov x$.  We will often use the following notation, in which $\thejipair$ comes from \quot{pair}:
\begin{equation}
\text{For }x\in\Jir L\text{, let }\jipair x\text{ or, if confusion threatens, }\ajipair L x\text{ stand for }(\locov x, x). 
\label{eq:pkmncShZLt}
\end{equation}
Note that the notation $\jipair x$ is meaningless if $x$ is not a  join-irreducible element.

For $x,y\in L$, $x\parallel y$ denotes that $x$ and $y$ are \emph{incomparable}, that is, $x\nleq y$ and $y\nleq x$. 
An element $x\in L$ is \emph{join-reducible} if it has at least two lower covers; the set of such elements is denoted by $\Jr L$, while $\Mr L$ stands for the set of (dually defined) \emph{meet-reducible} elements. Since $0\notin\Jir L\cup\Jr L$ and $1\notin \Mir L\cup \Mr L$, 
\begin{equation}
|L|=|\Jir L|+|\Jr L|+1 = |\Mir L|+|\Mr L|+1.
\label{eq:zsszmrdLjVntWz}
\end{equation}

By the classical structure theorem of finite distributive lattices, 
\begin{equation}
D\cong \Idl{\Jir D;\leq}
\label{eq:snCrtcHnnKks}
\end{equation}
for every such lattice $D$, where \quot{$\leq$} stands for the restriction of the lattice order of $D$ to $\Jir D$.  
For elements $a$ and $b$ of a finite lattice $L$, the least congruence collapsing $a$ and $b$ will be denoted by $\con(a,b)$ or, if ambiguity threatens, by $\con_L(a,b)$. 
In particular, for $b\in\Jir L$, $\con(\jipair b)=\con(\locov b,b)$. 

Let $\Edge L$ stand for the \emph{edge set} (set of edges) of $L$; that is, $\Edge L:=\{(a,b): a,b\in L$ and $a\prec b\}$. In other words, $\Edge L:=\{(a,b): [a,b]$ is a \emph{prime interval} of $L\}$.
We recall the following fact from Gr\"atzer \cite{GGfound}; see also the very first sentence of Gr\"atzer \cite{GGprimpers}: For every finite lattice $L$, 
\begin{equation}
\Jir{\Con L}=\set{\con(a,b): (a,b)\in\Edge L}. 
\label{eq:hmVkGrb}
\end{equation}

We define a quasiorder $\nu(L)$ on $\Edge L$ by letting  
\begin{equation}
(x,y) \leq_{\nu(L)}  (u,v) \defiff \con(x,y)\leq\con(u,v)
\label{eq:pPrdnjDgk}
\end{equation}
for $(x,y), (u,v)\in\Edge L$.
Combining \eqref{eq:snCrtcHnnKks}, \eqref{eq:hmVkGrb},  \eqref{eq:pPrdnjDgk}, and the fact that the congruence lattice of a lattice is always distributive, we obtain\footnote{Alternatively,  \eqref{eq:knRlhKrhkn} consists of (2.1) and a particular case of (2.8), both from Cz\'edli \cite{czg864}.}
that
\begin{equation}
\Con L\cong \Idl{\Jir{\Con L};\leq } \text{ \ and \ }
\Con L \cong \Idl{\Edge L; \nu(L)}
\label{eq:knRlhKrhkn}
\end{equation}
for every finite lattice $L$. Since \eqref{eq:knRlhKrhkn} in itself is not sufficient for our purposes, we need an additional concept. 
A subset $B$ of $\Edge L$ is a \emph{congruence-determining edge set} if for each $(x,y)\in\Edge L$ there is an edge $(u,v)\in B$ such that 
$(u,v)\equiv_{\nu(L)} (x,y)$, that is, $\con(u,v)=\con(x,y)$. Combining the facts gathered above, it is not hard to show (and it has been shown above (2.8) in Cz\'edli \cite{czg864}) that for every congruence-determining edge set $B$ of a finite lattice $L$, 
\begin{equation}
\Con L\cong \Idl{B;\nu(L)}.
\label{eq:brnkmR}
\end{equation}

For $a\leq b$ and $c\leq d$ in $L$, $[a,b]$ and $[c,d]$ are \emph{perspective intervals} if $[a,b]\pup [c,d]$, which means that
$b\vee c=d$ and $b\wedge c=a$, or $[a,b]\pdn [c,d]$, which means that $[c,d]\pup[a,b]$. We denote the \emph{perspectivity relation} on the set of intervals by $\pers$, so $[a,b]\pers [c,d]$ means that $[a,b]\pup [c,d]$ or $[a,b]\pdn [c,d]$. As it is well known, 
\begin{equation}
\text{if }[a,b]\pers [c,d]\text{, then }\con(a, b)=\con(c,d).
\label{eq:mHprsrTvkF}
\end{equation}

\begin{lemma}\label{lemma:oppedges}
In a finite lattice $L$, let $a,b,x\in L$ such that $a\parallel b$,  $b\prec a\vee b$, and 
$a\wedge b\prec x\leq a$. Then  $[a\wedge b, x]\pup[b,a\vee b]$ and $\con(a\wedge b,x)=\con(b,a\vee b)$.
\end{lemma}

\begin{proof} By virtue of \eqref{eq:mHprsrTvkF}, it suffices to show that $[a\wedge b, x]\pup[b,a\vee b]$.
If we had $x\leq b$, then $x\leq a\wedge b$ would be a contradiction. Thus $x\nleq b$. 
Since $a\wedge b=x\wedge a\wedge b \leq x\wedge b< x$ and $a\wedge b\prec x$, we obtain $x\wedge b=a\wedge b$.
From $x\nleq b$, we also obtain that $b< x\vee b$, whence $b<x\vee b\leq a\vee b$ and $b\prec a\vee b$ yield that $x\vee b=a\vee b$. Thus, $[a\wedge b, x]\pup[b,a\vee b]$, completing the proof.
\end{proof}

To formulate the result of Gr\"atzer \cite{GGprimpers} in a slightly weakened (but simpler) form, we recall the following lemma.

\begin{lemma}[Gr\"atzer \cite{GGprimpers}]\label{lemma:ggrprsprd}\ 

\textup{(A)} For a finite lattice $L$ and edges $(u,v),(p,q)\in \Edge L$, we have $(p,q)\leq_{\nu(L)}(u,v)$ or, equivalently, $\con(p,q)\leq \con (u,v)$ if and only if there is a number $k\in \Nnul$, a sequence $(x_0,y_0)$,  $(x_1,y_1)$, \dots, $(x_k,y_k)$ of edges, and a sequence   $[\dot x_1,\dot y_1]$, \dots, $[\dot x_k,\dot y_k]$ of intervals such that $(x_0,y_0)=(u,v)$, $(x_k,y_k)=(p,q)$, and for each $i\in\set{1,\dots,k}$, we have
$[x_{i-1},y_{i-1}]\pers [\dot x_i,\dot y_i]\supseteq [x_i,y_i]$.

\textup{(B)} The \quot{if} direction of  part \textup{(A)} remains valid under the weaker assumption that $x_i<y_i$ for $i\in\set{0,\dots, k}$, instead of requiring the pairs $(x_i,y_i)$ to be edges.
\end{lemma}

For $u$ in a quasiordered set $(A;\rho)$, $\idl u$ or $\idla A u$ 
denotes $\set{x\in A: x\leq_\rho u}$, while $\fil u$ or $\fila A u$ stands for $\set{x\in A: u\leq_\rho x}$. These notations apply in posets and lattices as well.  
For an edge $(a,b)\in \Edge L$ of a finite lattice $L$, let $\Foot a b$ or $\SFoot L a b$ denote the set of minimal elements of $\idl b\setminus \idl a$; we often denote its elements by $b'$ or $b''$.  The following paragraph, which is now well known, is basically taken from Day \cite[Page 71]{aday}.

For  $(a,b)\in\Edge L$, let $b'\in \Foot a b$. Clearly, $b'\neq 0$. If  $x$ and $y$ were distinct lower covers of $b'$, then (by the minimality of $b'$) they would belong to $\idl a$,  whence $b'=x\vee y\in\idl a$ would be a contradiction. Hence 
\begin{equation}\text{for every }(a,b)\in\Edge L,\text{ we have }\Foot a b\subseteq \Jir L.
\label{eq:hlKtngKGlmch}
\end{equation}
Note that in some cases \eqref{eq:hlKtngKGlmch} will be used only implicitly. For  $b'\in\Foot a b$, the covering relations $a\prec b$ and $\locov{b}'\prec b'$ easily imply that $[a,b]\pdn [\locov{b}',b']$. Hence \eqref{eq:mHprsrTvkF} implies $\con(\locov b',b')=\con(a,b)$. This equality,  \eqref{eq:hmVkGrb}, and \eqref{eq:hlKtngKGlmch} give that 
\begin{align}
\con(\jipair{b'})&=\con(a,b)\text{ for all }(a,b)\in\Edge L\text{ and }b'\in\Foot a b,
\label{eq:hlMpcLlskLmzs}
\\
\Jir{\Con L} &= \set{\con(\jipair x) :x\in\Jir L}
\text{, and}
\label{eq:mKflvGmgfll}
\\
\jipair{\Jir L}&:=\set{\jipair x: x\in \Jir L}\text{ is a congruence-determining edge set.}
\label{eq:mhvrVkGmT}
\end{align}
As \eqref{eq:mhvrVkGmT} indicates, $\Con L$ can be studied via $\Jir L$ rather than $\Edge L$, as demonstrated by, say,  Day \cite{aday}, Freese \cite{freese}, and Nation \cite[Theorem 316]{Nation}. Nevertheless, we usually rely on $\Edge L$, since the \quot{critical part} of $L$ often lacks join-irreducible elements and  $\Edge L$  offers better insight; for example, see the $H\cong H_2$ case in the proof of Lemma \ref{lemma:3lowcov}. We need the following lemma.

\begin{lemma}\label{lemma:diffeet} Given a finite lattice $L$, let $(u_1,v_1), (u_2,v_2)\in \Edge L$. Pick  $v_1'\in \Foot {u_1}{v_1}$ and $v_2'\in \Foot{u_2}{v_2}$. Assume that
either $u_1\leq u_2$ and $v_1\nleq v_2$, or that $v_1\leq u_2$. 
Then $v_1'\neq v_2'$.  
\end{lemma}

\begin{proof} Suppose, contrary to the assumptions, that $v_1'=v_2'$. For $i\in\set{1,2}$,  $v_i=u_i\vee v_i'$.
If $u_1\leq u_2$ but $v_1\nleq v_2$, then $v_1=v_1'\vee u_1=v_2'\vee u_1\leq v_2'\vee u_2=v_2$, which contradicts $v_1\nleq v_2$. If $v_1\leq u_2$, 
then $u_1\prec v_1\leq u_2$ implies $u_2=u_2\vee u_1$, whence
$v_2=u_2\vee v_2'=u_2\vee v_1'=u_2\vee u_1\vee  v_1'=u_2\vee v_1=u_2$ is a contradiction, completing the proof.
\end{proof}

The following lemma is straightforward, so we omit its proof.

\begin{lemma}\label{lemma:mslhGntCh}
Every meet-semilattice generated by three pairwise incomparable elements is isomorphic to one of those depicted in Figure \ref{figkorte}.
\end{lemma}


\begin{figure}[ht] 
\centerline{ \includegraphics[width=\figwidthcoeff\textwidth]{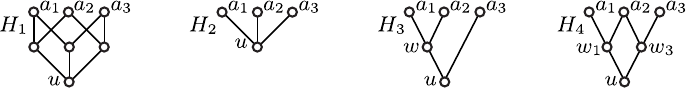}} 
\caption{The meet-semilattices generated by a three-element antichain}\label{figkorte}
\end{figure}

\begin{lemma}\label{lemma:nofourjred}
If $|\Jr L|\geq 4$ or $|\Mr L|\geq 4$ for a finite lattice $L$, then $\cd L\leq 4/64$. Similarly, if $L$ has an at-most $(|L|-5)$-element congruence-determining edge set, then $\cd L\leq 4/64$.
\end{lemma}

\begin{proof} Let $n:=|L|$. To verify the second part, let $B$ be an at-most $(n-5)$-element congruence-determining edge set. Then $|\Con L|\leq 2^{n-5}$ by \eqref{eq:brnkmR}. Hence $\cd L\leq 2^{n-5}/2^{n-1}=4/64$, as required. When proving the first part, we may, by duality\footnote{Since $\Con L=\Con{L\dual}$, the Duality Principle applies.}, assume $|\Jr L|\geq 4$. Then $\jipair{\Jir L}$, which is a congruence-determining edge set by \eqref{eq:mhvrVkGmT},  has at most $n-5$ elements by \eqref{eq:zsszmrdLjVntWz}. Hence the already established second part applies,  completing the proof of Lemma \ref{lemma:nofourjred}.
\end{proof}

With the notation introduced in \eqref{eq:pPrdnjDgk} and with $B\subseteq \Edge L$, we say that the quasiordered set $(B;\nu(L))$ is an \emph{antichain} if for every $x,y\in B$, $x\leq_{\nu(L)} y$ implies $x=y$. 

\begin{lemma}\label{lemma:lfkbbZmbt}
Assume that $L$ is an $n$-element lattice, $B\subseteq \Edge L$ is a congruence-determining edge set, $|B|\leq n-4$,  and $(B;\nu(L))$ is not an antichain. Then $\cd L\leq 6/64$.
\end{lemma}

\begin{proof} Pick $e,d\in B$ such that $d\leq_{\nu(L)} e$ and $d\neq e$. Then $B$ has exactly $2^{|B|-2}$ subsets $X$ with $e\in X$ but $d\notin X$. Using \eqref{eq:brnkmR} and the fact that these subsets $X$ are not ideals, 
\[  |\Con L|= |\Idl{B;\nu(L)}|\leq 2^{|B|} - 2^{|B|-2}=3\cdot 2^{|B|-2}\leq 3\cdot 2^{n-6}.
\]
Thus $\cd L\leq 3\cdot 2^{n-6} / 2^{n-1}=6/64$, which completes the proof of Lemma \ref{lemma:lfkbbZmbt}.
\end{proof}

Sometimes, the following statement is more convenient than Lemma \ref{lemma:lfkbbZmbt}.

\begin{lemma}\label{lemma:khVnMtTlkBr} Let $L$ be a finite lattice with $|\Jr L|\geq 3$, and let $(u_1,v_1), (u_2,v_2)\in \Edge L$ such that $\con(u_1,v_1)$ and $\con(u_2,v_2)$ are comparable. If either $u_1\leq u_2$ and $v_1\nleq v_2$ or $v_1\leq u_2$, then $\cd L\leq 6/64$.  
\end{lemma}

\begin{proof} Let $B:=\jipair{\Jir L}$. Then $B$ is a congruence-determining edge set by \eqref{eq:mhvrVkGmT}, and $|B|\leq n-4$ by \eqref{eq:zsszmrdLjVntWz}.
 Pick $v'_1\in\Foot{u_1}{v_1}$ and $v'_2\in\Foot{u_2}{v_2}$. Lemma  \ref{lemma:diffeet} and \eqref{eq:hlKtngKGlmch} imply that $v'_1\neq v'_2$ and $v'_1,v'_2\in\Jir L$. 
 Hence $\jipair{v'_1}$ and $\jipair{v'_2}$ are distinct members of $B$. 
Since $\con(u_1,v_1)$ and $\con(u_2,v_2)$ are comparable, so are $\jipair{v'_1}$ and $\jipair{v'_2}$ in $(B;\nu(L))$ by \eqref{eq:pPrdnjDgk} and \eqref{eq:hlMpcLlskLmzs}. Thus Lemma \ref{lemma:lfkbbZmbt} applies, completing the proof of Lemma \ref{lemma:khVnMtTlkBr}.
\end{proof}

For $2\leq k\in\Nplu$, we denote by $M_k$ the $(k+2)$-element lattice consisting of the elements $0$ and $1$, together with $k$ atoms; for $k\in\set{2,3,4}$, $M_k$ is drawn on the right of Figure \ref{figalma}.

\begin{lemma}\label{lemma:beffnachn} Let $L$ be a finite lattice with  $\cd L>6/64$.
For $3\leq k\in\Nplu$, let $S=\{u, a_1$, $\dots$, $a_k,v\}$ be a $(k+2)$-element sublattice of $L$ such that $u\prec a_i< v$ for all $i\in\set{1,\dots,k}$. (In particular, $S\cong M_k$.) Then $k+|\Jr L|\leq 4$. 
\end{lemma}

\begin{proof} Assume the hypotheses. For $i\in\set{1,\dots,k}$, pick an element $b_i\in L$ such that $a_i\leq b_i\prec v$, and  pick an element $a_i'\in\Foot u{a_i}$.  By Lemma \ref{lemma:diffeet}, $|\set{a_1',\dots,a_k'}|=k$. The dual of Lemma \ref{lemma:oppedges} implies that $\con(u,a_i)=\con(b_j,v)$ for all $i, j\in\set{1,\dots,k}$ with $i\neq j$. 
Hence $\con(u,a_1)=\con(b_3,v)=\con(u,a_2)=\dots=\con(u,a_k)$. Thus, by \eqref{eq:hlMpcLlskLmzs}, $\con(\jipair{a_1'})=\dots= \con(\jipair{a_k'})$. By \eqref{eq:pPrdnjDgk}, this means that $\jipair{a_i'}\equiv_{\nu(L)} \jipair{a_j'}$ for all $i,j\in\set{1,\dots,k}$. Combining this fact with \eqref{eq:mhvrVkGmT},  we obtain that $B:=\jipair{\Jir L}\setminus \set{\jipair{a_2'},\dots, \jipair{a_k'}}$ is a congruence-determining edge set. By $|\set{a_1',\dots,a_k'}|=k$ and \eqref{eq:hlKtngKGlmch}, $|B|=|\Jir L|-(k-1)$. 
Taking \eqref{eq:zsszmrdLjVntWz} into account and denoting $|L|$ and $|\Jr L|$ by $n$ and $j$, respectively, we have $|B|=n-1-j-(k-1)=n-j-k$.
If we had $k+j\geq 5$, then $|B|\leq n-5$ and Lemma \ref{lemma:nofourjred} would imply $\cd L\leq 4/64$, which is a contradiction. 
Therefore, the integer $k+j$ is at most $4$, as required.  The proof of Lemma \ref{lemma:beffnachn} is complete. 
\end{proof}

For a sublattice $S$ of a finite lattice $L$, we say that $L$ is a \emph{dismantlable extension} of $S$ if there are sublattices $T_0, T_1,\dots, T_k$  for some $k\in\Nnul:=\Nplu\cup\set{0}$ such that
\begin{equation}
S=T_0\subset T_1\subset\dots\subset T_{k-1}\subset T_k=L
\text{ and }|T_i\setminus T_{i-1}|=1\text{ for }i\in\set{1,\dots,k}.
\label{eq:dismnTlb}
\end{equation}
Recall from Cz\'edli \cite{czgcdofS} that for any sublattice $S$ of a finite lattice $L$,
\begin{equation}
\text{if }L\text{ is a dismantlable extension of }S\text{, then }\cd L\leq \cd S.
\label{eq:dsmntlExtsm}
\end{equation}

We present the following lemma, as it will be used repeatedly.

\begin{lemma}\label{lemma:pdgbpLsT} If $S$ is a nonempty subset of a finite lattice $L$, then the following two assertions hold.
\begin{enumerate}
\item\label{lemma:apdgbpLsT} 
If  $\Mr L\subseteq S$ and $\Jr L\subseteq S$, then $S$ is a sublattice of $L$ and $\cd L\leq \cd S$. 
\item\label{lemma:bpdgbpLsT} If $\cd L>4/64$, $|S\cap \Jr  L| \geq 3$, and $|S\cap \Mr L| \geq 3$, then  $S$ is a sublattice of $L$, $\Mr L\subseteq S$, $\Jr L\subseteq S$, $|\Mr L|=|\Jr L|=3$, and $\cd L\leq \cd S$.
\end{enumerate}
\end{lemma}

\begin{proof} Assume that $\Mr L\subseteq S$ and $\Jr L\subseteq S$. We need to show that for every $x,y\in S$, the join $x\vee_L y$ (taken in $L$) is in $S$. If $x\parallel y$, then $x\vee_L y\in\Jr L\subseteq S$. For $x\nparallel y$, $x\vee_L y\in\set{x,y}\subseteq S$. Hence $S$ is closed with respect to joins. Thus $S$ is a sublattice of $L$ by the duality principle. Clearly, $L$ is a dismantlable extension of $S$. Hence, by \eqref{eq:dsmntlExtsm}, $\cd L\leq \cd S$. This proves Part \eqref{lemma:apdgbpLsT}.

To prove Part \eqref{lemma:bpdgbpLsT}, observe that $\cd L>4/64$ and Lemma \ref{lemma:nofourjred} imply $|\Jr L|\leq 3$ and $|\Mr L|\leq 3$. These inequalities, $|S\cap \Jr  L| \geq 3$, and $|S\cap \Mr L| \geq 3$ yield that $\Mr L\subseteq S$ and $\Jr L\subseteq S$. Hence the already proven Part \eqref{lemma:apdgbpLsT} implies Part \eqref{lemma:bpdgbpLsT} of Lemma \ref{lemma:pdgbpLsT}.
\end{proof}

The following lemma will also be needed.

\begin{lemma}\label{lemma:circedgglu}
Assume that $H$ and $K$ are finite lattices,   $K\in\Circ k$ with $4\leq k\in\Nplu$, and $L$ is an edge gluing of $H$ and $K$ (in either order). Then 
\begin{enumerate}
\item\label{lemma:acircedgglu} 
$\cd L \leq   \cd H /2$, and  
\item\label{lemma:bcircedgglu}
$\cd L =\cd H/2$ if and only if $k=4$ (that is, if $K=B_4$).   
\end{enumerate}
\end{lemma}

\begin{proof} 
We know from Cz\'edli \cite[Lemma 4]{czg864}
that $k=4$ implies $\cd L=\cd H/2$. 

Next, let $k=5$, that is, $K=N_5$; we need to show that $\cd L<\cd H/2$. By duality, we can assume that $N_5$ is the \quot{upper lattice} in the gluing. As in Figure \ref{figalma}, $(a,c)\in\Edge{N_5}$ denotes the so-called \emph{monolith edge} of $N_5$. Since $c\in L$ is outside $H$, $L':=L\setminus\set c$ is an edge gluing of $H$ and $B_4=N_5\setminus \set c$.
We already know that $\cd {L'}=\cd H/2$, whence it suffices to show that $\cd L<\cd{L'}$. 
Let $B':=\jipair{\Jir{L'}}$ and $B:=\jipair{\Jir{L}}$.
Since $\Jir{L} = \Jir{L'}\cup\set{c}$, we have
$B=B'\cup\set{\jipair c} =B'\cup\set{(a,c)}$.
We claim that for any $X$, 
\begin{equation}
\text{if }X\in\Idl{B;\nu(L)} ,\text{ then }
X\cap B'\in\Idl{B';\nu(L')}.
\label{eq:tbmCtmTkrlM}
\end{equation}
For $\jipair u\in X\cap B'$ and $\jipair v\in B'$, assume that $ \jipair v\leq_{\nu(L')} \jipair u$. 
Note at this point that none of $u$ and $v$ equals the join-reducible element $1_L=1_{L'}$, whence $\jipair u$ and $\jipair v$ mean the same in $L$ as in $L'$.
The inequality $ \jipair v\leq_{\nu(L')} \jipair u$ means that $\jipair v=(\locov v,v)\in \con_{L'}(\locov u,u)=\con_{L'}(\jipair u)$. By Lemma \ref{lemma:ggrprsprd}(A), this membership is witnessed by sequences of edges and intervals of $L'$. Even though an edge $(x_i,y_i)\in \Edge{L'}$ need not be an edge of $L$, Lemma \ref{lemma:ggrprsprd}(B) implies $\jipair v\in \con_{L}(\jipair u)$, that is, $\jipair v\leq_{\nu(L)}\jipair u$. Using $X\in\Idl{B;\nu(L)}$, we obtain $\jipair v\in X$. Hence, $\jipair v\in X\cap B'$, proving \eqref{eq:tbmCtmTkrlM}.

By  \eqref{eq:tbmCtmTkrlM}, the assignment $X\mapsto X\cap B'$ is a function 
\begin{equation*}f\colon \Idl{B;\nu(L)} \to \Idl{B';\nu(L')}.
\end{equation*} 
Since $B\setminus B'=\set{\jipair c}$, each $Y\in \Idl{B';\nu(L')}$ has at most two $f$-preimages: $Y$ and $Y\cup\set{\jipair c}$. 
Therefore,
\begin{equation}
|\Idl{B;\nu(L)}|\leq 2\cdot|\Idl{B';\nu(L')}|,
\label{eq:amgrJjPlz}
\end{equation}
and the inequality in \eqref{eq:amgrJjPlz} 
\begin{equation}
\text{is an equality }\iff \text{each }Y\in\Idl{B';\nu(L')}\text{ has exactly two }f\text{-preimages.}
\label{eq:bmgrJjPlz}
\end{equation}
Now, consider the element $b\in N_5$ shown in Figure \ref{figalma}, and pick $Y\in \Idl{B';\nu(L')}$ such that $\jipair b\in Y$; for example, we may take $Y:=B'$. Here $\jipair b=(\locov b,b)=(0_{N_5},b)$.
From  $[\locov b,b]\pup [a,1]$ and \eqref{eq:mHprsrTvkF}, we obtain $(a,1)\in \con_L(\jipair b)$.
Since the blocks of any lattice congruence are convex subsets, $a\leq c\leq 1$ and the convexity of the $\con_L(\jipair b)$-block of $a$ imply $\jipair c=(a,c)\in \con_L(\jipair b)$. Thus $\con_L(\jipair c)\leq \con_L(\jipair b)$, which means that  $\jipair c\leq_{\nu(L)}\jipair b$. Hence $Y$, which does not contain $\jipair c$, cannot be an $f$-preimage of $Y$, whence the only $f$-preimage of $Y$ is $Y\cup\set{\jipair c}$. Thus \eqref{eq:amgrJjPlz} and \eqref{eq:bmgrJjPlz} imply $|\Idl{B;\nu(L)}| < 2\cdot|\Idl{B';\nu(L')}|$. This inequality turns into $|\Con{L}| < 2\cdot |\Con{L'}|$ by \eqref{eq:brnkmR}. Therefore, using $|L|-1=|L'|$,
\begin{equation*}
\cd L=|\Con L|/2^{|L|-1} < 2\cdot |\Con{L'}|/2^{|L|-1}=|\Con{L'}|/2^{|L'|-1}=\cd{L'},
\end{equation*}
proving the lemma for $k=5$.

Finally, assume that $k\geq 6$. In $K$, we can find a $(k-5)$-element set $U$ of doubly irreducible elements 
such that $U\cap H=\emptyset$ and $K':=K\setminus U\in \Circ 5$. Then $L$ is a dismantlable extension of its sublattice $L':=L\setminus U$, and $L'$ is an edge gluing of $H$ and $K'$. Since $K'\in \Circ 5$ and thus $K'\cong N_5$, we have already seen that $\cd {L'}<\cd H/2$. Therefore, using that $\cd L\leq \cd{L'}$ by \eqref{eq:dsmntlExtsm}, we have $\cd{L}< \cd H/2$, completing the proof of Lemma \ref{lemma:circedgglu}.
\end{proof}

We conclude this section with the following convention, which will be used in the diagrams occurring in the paper.

\begin{convention} If a diagram contains edges of different line styles, then 
\begin{itemize}
\item a thin edge directed upward from a vertex $x$ to a vertex $y$ represents $x<y$,
\item a thick edge from $x$ to $y$  represents $x\prec y$, and
\item a dotted edge from $x$ to $y$ represents $x\leq y$. 
\end{itemize}
In diagrams containing only thin edges, an edge from $x$ upwards to $y$ can  represent either a covering $x\prec y$ (as usual) or, when only a few elements of a larger lattice $L$ are depicted, $x<y$. 
  If we label an edge $(a,b)$ with a black Greek letter,  then we may label an edge $(x,y)$ or an interval $[x,y]$ with the same red (grey in non-color prints) Greek letter whenever $\con(x,y)=\con(a,b)$.
\end{convention}

\section{The case of a three-element antichain}

As the lower covers of an element form an antichain, the following lemma addresses a particular case in this section. For $M_3$, see Figure \ref{figalma}. 

\begin{lemma}\label{lemma:3lowcov}
Assume that $L$ is a finite lattice with $6/64<\cd L$, $v\in L$, and $v$ has at least three lower covers. Then $\cd L\leq 8/64$, and $v$ has exactly three lower covers. Furthermore, denoting by $u$ the meet of these lower covers, the interval $[u,v]$ is isomorphic to $M_3$.
\end{lemma}

\begin{proof} Lemma 6 of Cz\'edli \cite{czg864} states that if $L$ has a three-element antichain, then $\cd L\leq 8/64$. Hence $\cd L\leq 8/64$, as required.  (A weaker statement, Lemma 5 of \cite{czg864} would also suffice for this conclusion.)

Let $n:=|L|$.
Our plan is to construct a small congruence-determining edge set. We start with $B_0:=\jipair{\Jir L}$. Using \eqref{eq:mhvrVkGmT} and  $0,v\notin \Jir L$, we obtain that 
\begin{equation}
B_0 \text{ is a congruence-determining edge set and }|B_0|\leq n-2.
\label{eq:mrtzrsgrtkk}
\end{equation}

Let $a_1,a_2,a_3$ be distinct lower covers of $v$, and denote $a_1\wedge a_2\wedge a_3$ by $u$. Let $H$ stand for the meet-subsemilattice generated by $\set{a_1,a_2,a_3}$. By Lemma \ref{lemma:mslhGntCh}, $H$ is isomorphic to one of the four meet-semilattices drawn in Figure \ref{figkorte}. 
Depending on $H$,  there are four cases to consider.

First, assume that $H\cong H_1$; see Figure \ref{figkorte}. Then $H\setminus\set{a_1, a_2, a_3}$ is a four-element subset of $\Mr L$, whence $|\Mr L|\geq 4$. Thus Lemma \ref{lemma:nofourjred} implies $\cd L\leq 4/64$, which contradicts $6/64<\cd L$. This excludes the case  $H\cong H_1$.


\begin{figure}[ht] 
\centerline{ \includegraphics[width=\figwidthcoeff\textwidth]{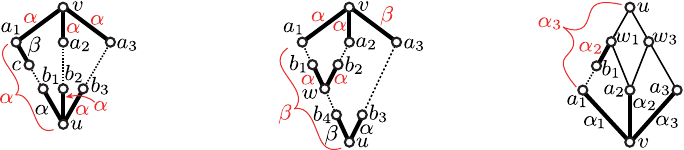}} 
\caption{Illustrations for the proof of Lemma \ref{lemma:3lowcov}}\label{figmalna}
\end{figure}

\logicbreak
Next, assume that $H\cong H_2$. Then  $S:=\set{u,a_1,a_2,a_3,v}$ is a sublattice of $L$, and it is isomorphic to $M_3$. We claim that $S$ is a cover-preserving sublattice of $L$, that is, $u\prec a_i\prec v$ for $i\in\set{1,2,3}$. 
Suppose, for a contradiction, that $S$ is not a cover-preserving sublattice. Then, by symmetry, we may assume that there are elements $b_1$ and $c$ such that $u\prec b_1\leq c\prec a_1$. 
The situation is visualized in the left part of Figure \ref{figmalna}; note that, here and later in some illustrations, if a diagram is not labeled by $L$, then $L$ may contain elements 
not shown. For $i\in\set{2,3}$, pick an element $b_i\in L$ such that $u\prec b_i\leq a_i$. 
Since $a_i\wedge a_j=u$, it follows easily that $b_i\parallel b_j$ for distinct $i,j\in\set{1,2,3}$ and, furthermore, $a_1\parallel b_2$ and $a_1\parallel b_3$. 
Denote $\con(u,b_1)$ by $\alpha$. Applying Lemma \ref{lemma:oppedges} and \eqref{eq:mHprsrTvkF} several times, we obtain that $\alpha=\con(a_2,v)=\con(u,b_3)=\con(a_1,v)=\dots$, that is, $\alpha=\con(u,b_i)=\con(a_j,v)$ for all $i,j\in\set{1,2,3}$.

Let $\beta:=\con(c,a_1)$.
Since $[a_2,v]\pdn [u,a_1]$, we have $(u,a_1)\in\alpha$ by \eqref{eq:mHprsrTvkF}. The blocks of a lattice congruence are always convex sublattices. Hence the $\alpha$-block of $a_1$ contains $c$.
Thus $(c,a_1)\in \alpha$, whence $\beta\leq \alpha$.

Note that our considerations like the one in the previous paragraph will be presented more concisely later; for example, in the form
\quot{$[a_2,v]\pdn [u,a_1] \supseteq [c,a_1] $ implies $\beta=\con(c,a_1)\leq \con(a_2,v)=\alpha$}.
  
For $i\in\set{1,2,3}$, pick an element $b_i'\in \Foot u {b_i}$, and pick an $a_1'\in \Foot c {a_1}$. 
Since $\con(u,b_i)=\alpha$ for every $i\in\set{1,2,3}$, each $\con(\jipair{b_i'})$ is $\alpha$ by \eqref{eq:hlMpcLlskLmzs}. Combining this with \eqref{eq:mrtzrsgrtkk}, we obtain that $B_1:=B_0\setminus\set{\jipair{b_2'},\jipair{b_3'}}$ is a congruence-determining edge set. 
By Lemma \ref{lemma:diffeet}, 
\begin{equation*}
|\set{\jipair{a_1'},\jipair{b_1'},\jipair{b_2'},\jipair{b_3'}}| = |\set{a_1',b_1',b_2',b_3'}|=4.
\end{equation*} 
This equality, \eqref{eq:hlKtngKGlmch}, and \eqref{eq:mrtzrsgrtkk} give that $|B_1|=|B_0|-2\leq n-4$ and that $\jipair{a_1'}$ and $\jipair{b_1'}$ are distinct elements of $B_1$. Since $\con(\jipair{a_1'})=\beta$ and $\con(\jipair{b_1'})=\alpha$ by \eqref{eq:hlMpcLlskLmzs},  $\beta\leq \alpha$ turns into $\jipair{a_1'}\leq_{\nu(L)}\jipair{b_1'}$, whence $(B_1,\nu(L))$ is not an antichain. Thus Lemma \ref{lemma:lfkbbZmbt} yields $\cd L\leq 6/64$, which is a contradiction.  This proves that $S=\set{u,a_1,a_2,a_3,v}$ is a cover-preserving sublattice of $L$.  

Next, we show that $[u,v]=S$. Suppose, for a contradiction, that $S\subset [u,v]$, and pick a minimal element $a_4\in[u,v]\setminus S$. Since $S$ is a cover-preserving sublattice, we have $a_4\parallel a_i$ for $i\in \set{1,2,3}$. Thus, for $i\in\set{1,2,3}$,  the coverings $u\prec a_i$ and $a_i\prec v$ yield that   $a_i\vee a_4=v$ and $a_i\wedge a_4=u$. Hence, $S':=S\cup\set{a_4}$ is a sublattice of $L$, and $S'\cong M_4$. The minimality of $a_4$ implies $u\prec a_4$. 
Thus $4+|\Jr L|\leq 4$ by Lemma \ref{lemma:beffnachn}, and so $|\Jr L|=0$. This contradicts $v\in\Jr L$, whence $[u,v]=S$. Therefore $[u,v]=S\cong{M_3}$.

Finally, let $d$ be a lower cover of $v$ such that $d\notin\set{a_1,a_2}$, and let $w$ and $T$ denote $a_1\wedge a_2\wedge d$ and the sublattice generated by $\set{a_1,a_2,d}$, respectively. Since $(d,w,T)$ and $(a_3,u,S)$ play the same role, $T=[w,u]\cong M_3$. Hence 
$d\in T=[w,v]=[a_1\wedge a_2,v]=[u,v]=S$. This implies $d=a_3$, whence $v$ has only three lower covers.  Consequently, whenever $H\cong H_2$, the statement of Lemma \ref{lemma:3lowcov} holds.

\logicbreak

Next, we deal with the case when $H\cong H_3$. Apart from indexing, $a_1\wedge a_3=a_2\wedge a_3=u<a_1\wedge a_2=:w$;  see the middle of Figure \ref{figmalna}. Pick elements $b_1,\dots, b_4\in L$ such that 
$w\prec b_1\leq a_1$, $w\prec b_2\leq a_2$, $u\prec b_3\leq a_3$, and $u\prec b_4\leq w$. Furthermore, select    
$b_1'\in\Foot w{b_1}$, $b_2'\in\Foot w{b_2}$, $b_3'\in\Foot u{b_3}$, and $b_4'\in\Foot u{b_4}$. 
Define $\alpha:=\con(u,b_3)$ and $\beta:=\con(u,b_4)$. 
Using Lemma \ref{lemma:oppedges} and \eqref{eq:mHprsrTvkF} three times,
we obtain for $i\in\set{1,2}$ that $\con(w,b_i)=\con(a_{3-i},v)=\alpha$ and $\con(a_3,v)=\beta$.
Hence, by \eqref{eq:hlMpcLlskLmzs}, $\con(\jipair{b_i'})=\alpha$ for $i\in\set{1,2,3}$. 
Note at this point that, for brevity, the application of  \eqref{eq:hlMpcLlskLmzs} will often be implicit.
Since $[a_3,v]\pdn [u,a_1]\supseteq[w,b_1]$, we have $\alpha=\con(w,b_1)\leq \con(a_3,v)=\beta$.
 Hence,  $\con(\jipair{b'_3})=\alpha\leq \beta=\con(\jipair{b'_4})$. That is,  $\jipair{b'_3} \leq_{\nu(L)} \jipair{b'_4}$.

Since $b_1\leq b_3$ would lead to $w\leq b_1=b_1\wedge b_3\leq a_1\wedge a_3= u$, while  $b_3\leq b_1$ to $b_3= b_1\wedge b_3\leq  a_1\wedge a_3=u$, we have $b_1\parallel b_3$. Similarly, $b_2\parallel b_3$. 
These facts, the obvious comparabilities and incomparabilities like $b_1\parallel b_2$ and $b_4\leq w$, and Lemma \ref{lemma:diffeet}
imply that 
\begin{equation}
|\set{\jipair{b_1'},\jipair{b_2'},\jipair{b_3'},\jipair{b_4'}}|=|\set{b_1',b_2',b_3',b_4'}|=4.
\label{eq:KtzrkBnhGdj}
\end{equation}
Combining \eqref{eq:hlKtngKGlmch},
the fact that $\con(\jipair{b_i'})=\alpha$ for $i\in\set{1,2,3}$,
 \eqref{eq:mrtzrsgrtkk}, and \eqref{eq:KtzrkBnhGdj}, it follows  that    $B_1:=B_0\setminus\set{\jipair{b'_1}, \jipair{b'_2}}$ is an at-most $(n-4)$-element congruence-determining edge set containing $\jipair{b_3'}$ and $\jipair{b_4'}$.
Thus $\jipair{b'_3} \leq_{\nu(L)} \jipair{b'_4}$ and Lemma \ref{lemma:lfkbbZmbt} imply $\cd L\leq 6/64$, a contradiction. Therefore the case $H\cong H_3$ cannot occur.

\logicbreak

Finally, we deal with the case when $H\cong H_4$; see Figure \ref{figkorte}. Our aim is to exclude this case.
Since $\Con{L\dual}=\Con L$, we can work in $K:=L\dual$, where the join-reducible elements are easier to handle. All order-theoretic symbols  such as  $\leq$, $\vee$, $\prec$, etc.\ will be understood in $K$. 
In particular, $v\prec a_i$ for $i\in\set{1,2,3}$, and $w_i:=a_i\vee a_2<a_1\vee a_3=:u$ for $i\in\set{1,3}$. 
The situation is shown in the right part of Figure~\ref{figmalna}. Clearly,  $|\Jr K|\geq 3$.
Take an element $b_1\in K$ such that $a_1\leq b_1\prec w_1$, and let $\alpha_i:=\con(v,a_i)$ for $i\in\set{1,2,3}$. By  \eqref{eq:mHprsrTvkF} and the dual of Lemma \ref{lemma:oppedges}, 
$\con(b_1,w_1)=\alpha_2$. Since $[v,a_3]\pup[a_1,u]\supseteq[b_1,w_1]$, we obtain that 
$\con(v,a_2)=\alpha_2=\con(b_1,w_1) \leq \con(v,a_3)$. Hence Lemma \ref{lemma:khVnMtTlkBr} implies 
$\cd L=\cd K\leq 6/64$, which is a contradiction excluding the case $H\cong H_4$.
Thus the proof of Lemma \ref{lemma:3lowcov} is complete.
\end{proof}

\begin{lemma}\label{lemma-3a-cha}
Let $\set{a_1,a_2,a_3}$ be a three-element antichain in a finite lattice $L$ with $\cd L>6/64$, and
assume that $\Core L\ncong M_3$. Then the sublattice generated by  $\set{a_1,a_2,a_3}$ is isomorphic to $N_{5,5}$ (depicted on the left of Figure \ref{figdio}).
\end{lemma}

\begin{proof}   Let $n:=|L|$, $u:=a_1\wedge a_2\wedge a_3$, and $v:=a_1\vee a_2\vee a_3$. 
Denote by $S^\wedge$ and $S^\vee$ the meet-subsemilattice and the join-sub\-semilattice generated by $\set{a_1,a_2,a_3}$. By Lemma \ref{lemma:mslhGntCh} and its dual statement, $S^\wedge \in\set{H_1,\dots, H_4}$ and 
$S^\vee\in \set{H_1\dual,\dots, H_4\dual}$ (both understood up to isomorphism).  
There are several cases to consider.

\logicbreak
Case 1: We assume that $S^\wedge\cong H_1$ or  $S^\vee\cong H_1\dual$. 
Then $|\Mr L|\geq 4$ or $|\Jr L|\geq 4$, which implies  $\cd L\leq 4/64$ by
Lemma \ref{lemma:nofourjred}, contradicting the assumption $\cd L>6/64$. Therefore, Case 1 cannot occur; that is, $S^\wedge\ncong H_1$ and  $S^\vee\ncong H_1\dual$.

\logicbreak
Case 2:  $S^\wedge\cong H_2$ or  $S^\vee\cong H_2\dual$. 
Then, by duality, we may assume that  $S^\vee\cong H_2\dual$. We have that $v:=a_1\vee a_2=a_1\vee a_3=a_2\vee a_3$.
For $i\in\set {1,2,3}$, pick an element $b_i\in L$ such that $a_i\leq b_i\prec v$. These elements are distinct lower covers of $v$; indeed, if we had, say, $b_1=b_2$, then $v=a_1\vee a_2\leq b_1\vee b_2=b_1<v$ would be a contradiction. 
Let $u_0:=b_1\wedge b_2\wedge b_3$. By Lemma \ref{lemma:3lowcov}, $[u_0,v]\cong M_3$.   In particular, $u_0\prec b_i\prec v$ for $i\in\set{1,2,3}$. 
Hence, it follows from Lemma \ref{lemma:beffnachn} and its dual that $|\Jr L|\leq 1$ and $|\Mr L|\leq 1$. Combining this with $u_0\in\Mr L$ and $v\in \Jr L$, we have that $\Mr L=\set{u_0}$ and $\Jr L=\set v$.
If there were an element $x\in L$ with $x\parallel u_0$ or $x\parallel v$, then $x\wedge u_0\in \Mr L\setminus \set{u_0}=\emptyset$ or  $x\vee v\in \Jr L\setminus \set v=\emptyset$, respectively, would follow, which is impossible. Hence no such $x$ exists. Therefore, $u_0,v\in\Nar L$ and $L=\idl {u_0}\gsum[u_0,v]\gsum \fil v$. 
This glued sum decomposition, together with $\Mr L=\set{u_0}$,  $\Jr L=\set v$, and the already established $[u_0,v]\cong M_3$ implies that $\idl{u_0}$ and $\fil v$ are chains and, moreover, that $\Core L=[u_0,v]\cong M_3$.  This contradicts our assumption and excludes Case 2.

\logicbreak
Case 3: $S^ \wedge\cong H_4$ or $S^\vee\cong H_4\dual$. Duality allows us to assume $S^\vee\cong H_4\dual$.
Up to a permutation of the subscripts, 
$b_1:=a_1\vee a_2<a_1\vee a_3=:v$,  $b_3:=a_2\vee a_3<v$, and $v=b_1\vee b_3$. 
We know from Cases 1 and 2 that $S^\wedge \ncong H_1$ and $S^\wedge \ncong H_2$. Thus, by Lemma \ref{lemma:mslhGntCh}, $S^\wedge \cong H_3$ or $S^\wedge \cong H_4$.

First, to derive a contradiction, we assume that $S^\wedge \cong H_3$. 
Let $c$ be the unique element of $S^\wedge\setminus\set{u,a_1,a_2,a_3}$, where $u$ still denotes the smallest element of $S^\wedge$. 
Apart from symmetry, $c$ is either $a_1\wedge a_2$ or $a_1\wedge a_3$. 

Suppose, for a contradiction, that $c=a_1\wedge a_3$; see the lattice diagram in the middle of Figure \ref{figdio}. Since  $c\parallel a_2$ (in $S^\wedge$, and hence also in $L$),   $c\vee a_2\in \Jr L$. Given that $c\vee a_2\leq a_1\vee a_2=b_1$,  
$c\vee a_2\leq a_3\vee a_2=b_3$, and $b_1\parallel b_3$,  we obtain that $c\vee a_2\notin \set{b_1,b_3,v}$. Hence $\set{c\vee a_2, b_1, b_3,v}$ is a four-element subset of $\Jr L$, whence 
Lemma \ref{lemma:nofourjred} implies $\cd L\leq 4/64$,  a contradiction. Thus  $c\neq a_1\wedge a_3$. 

Therefore $c=a_1\wedge a_2$, as shown on the right of Figure \ref{figdio}. Then $b_2:=b_1\wedge b_3\geq a_2$, whence $\set{b_2,c,u}$ is a 3-element subset of $\Mr L$. Therefore, since $\set{b_1,b_3,v}$ is a 3-element subset of $\Jr L$, Lemma \ref{lemma:pdgbpLsT}\eqref{lemma:bpdgbpLsT} 
yields that $S:=\{u,v,a_1,a_2,a_3,b_1,b_2,b_3,c\}$ is a sublattice of $L$ and $\cd L\leq \cd S$. Note 
that---depending on whether $a_2$ equals $b_2$ or not---$|S|$ is $8$ or $9$.
We claim that
\begin{equation}
S\text{ is correctly depicted on the right of Figure \ref{figdio}.}
\label{eq:Scorrdep}
\end{equation}
Since $S$ being a sublattice of $L$ does not by itself imply \eqref{eq:Scorrdep}, we provide some details here. (Later, in analogous situations, we will usually omit such details and simply say that 
the argument would be similar to the one proving \eqref{eq:Scorrdep}.) 
The \quot{$\Mr L\subseteq S$, $\Jr L\subseteq S$, and  $|\Mr L|=|\Jr L|=3$ part} of  Lemma \ref{lemma:pdgbpLsT}\eqref{lemma:bpdgbpLsT}, together with our assumption on $S^\vee$ and $S^\wedge$, suffices  to verify \eqref{eq:Scorrdep} as follows.  
From  $S^\vee$, we know that $b_1\parallel a_3$. Hence $b_1\wedge a_3\in \Mr L=\set{b_2,c,u}$.  Since $c\nleq a_3$ (even within $S^\wedge$) and $b_2\nleq a_3$ (as otherwise $c\leq a_2\leq b_2\leq a_3$ would give $c\leq a_3$), neither $c$ nor  $b_2$ is $b_1\wedge a_3$. Thus $b_1\wedge a_3\in \set{b_2,c,u}\setminus \set{c,b_2}$, whence $b_1\wedge a_3=u$. 
Similarly, since  $c\parallel a_3$ implies that $c\vee a_3\in \Jr  L=\set{b_1,b_3,v}$, only $b_3$ and $v$ among the elements of $\Jr L$ are common upper bounds of $c$ and $a_3$,   and since moreover $b_3\leq v$, we obtain $c\vee a_3=b_3$. Finally, since the lattice operations are isotone (in other words, order-preserving), the already established equalities combined with our knowledge of $S^\vee$ and $S^\wedge$ yield \eqref{eq:Scorrdep}.

Working in $S$, $[u,a_3]\pup [c,b_3]\supseteq [c,a_2]$ yields $\con_S(c,a_2)\leq \con_S(u,a_3)$. Thus Lemma \ref{lemma:khVnMtTlkBr} together with \eqref{eq:Scorrdep} implies $\cd L\leq 6/64$, which is a contradiction. Consequently, $S^\wedge\ncong H_3$.

Second, in pursuit of a contradiction, assume  that $S^\wedge \cong H_4$. For a permutation $p$ of $\set{1,2,3}$, $a_{p(1)}\wedge a_{p(2)}\in S^\wedge$ and $a_{p(2)}\wedge a_{p(3)}\in S^\wedge$ are incomparable elements. Thus, 
 $(a_{p(1)}\wedge a_{p(2)})\vee (a_{p(2)}\wedge a_{p(3)})$---which is in $\idl {a_{p(2)}}$---and $b_1, b_3, v$---which are out of $\idl {a_{p(2)}}$---are four different elements of $\Jr L$. Hence, Lemma \ref{lemma:nofourjred} implies $\cd L\leq 4/64$, which contradicts the assumption $\cd L>6/64$. Consequently, $S^\wedge \ncong H_4$. Since every possibility for $S^\wedge$ has led to a contradiction, we conclude that Case 3 cannot occur.

\logicbreak
Next, having excluded Cases 1, 2, and 3 for \emph{every} three-element antichain, it follows that  
$S^\wedge \cong H_3$ and $S^\vee\cong H_3\dual$. 
Without loss of generality (after permuting the subscripts, if necessary),  we may assume that $b:=a_1\vee a_2<v$, and
\begin{equation}
\text{either }a_2\wedge a_3>u\text{ or }a_1\wedge a_2>u.
\label{eq:fFjmlVgysBck}
\end{equation}


\begin{figure}[ht] 
\centerline{ \includegraphics[width=\figwidthcoeff\textwidth]{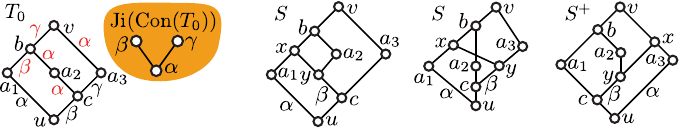}} 
\caption{Illustrations of the $S^\wedge \cong H_3$ and $S^\vee\cong H_3\dual$ case}\label{fignaspolya}
\end{figure}

To exclude the first possibility in \eqref{eq:fFjmlVgysBck} by contradiction, let us suppose that $c:=a_2\wedge a_3>u$; see the left part of Figure \ref{fignaspolya}. 
We claim that, under this assumption,
\begin{equation}
a_1\vee c= b\text{ \ and \ }b\wedge a_3=c.
\label{eq:tdVntXhRtskHg}
\end{equation}
By duality, it suffices to deal with the first equality. Let $x:=a_1\vee c$. To obtain a contradiction, suppose $x<b$; see the two middle lattice diagrams of Figure \ref{fignaspolya}.
Since  $x\leq a_2$ would lead to $a_1\leq a_2$, 
$x\leq a_3$ to $a_1\leq a_3$, $a_2\leq x$ to $b\leq x$, and $a_3\leq x$ to $v\leq x$, we obtain that  $\set{x,a_2,a_3}$ is a 3-element antichain. 
Since  Cases 1, 2, 3 (for every 3-element antichain) have already been excluded, the meet-semilattice generated by $\set{x,a_2,a_3}$ is isomorphic to $H_3$. 
Thus, using  
$a_2\wedge a_3=c=x\wedge c=  x\wedge a_2\wedge a_3$, 
it follows that exactly one of $x\wedge a_2$ and $x\wedge a_3$ is strictly larger than $c$, while the other equals $c$. 
The corresponding two subcases are drawn in 
the two middle lattice diagrams
of Figure \ref{fignaspolya}.
The treatment of these two subcases is the same. Denote by $y$ the larger of $x\wedge a_2$ and $x\wedge a_3$; the other one is $c$. 
Since $a_1\nleq a_2$ and $a_1\nleq a_3$ , we have that $a_1\nleq y$.
Discarding its previous meaning, define $S:=\{u,v,a_1,a_2,a_3,b,c, x,y\}$, which is a subposet of $L$. 
Similarly to the argument used to verify \eqref{eq:Scorrdep}, we obtain from Lemma \ref{lemma:pdgbpLsT}\eqref{lemma:bpdgbpLsT} that $S$ is a sublattice of $L$, that it is correctly depicted, and that $\cd L\leq \cd S$. 
In $S$, let $\alpha:=\con_S(u,a_1)$ and $\beta:=\con_S(c,y)$. 
Again in $S$, we have that $[u,a_1]\pup[c,x]\supseteq [c,y]$, whence $\beta\leq \alpha$.
Thus, Lemma \ref{lemma:khVnMtTlkBr} yields that $\cd S\leq 6/64$, whence  
$\cd L\leq \cd S\leq 6/64$ is a contradiction proving \eqref{eq:tdVntXhRtskHg}.

It follows from  $S^\wedge \cong H_3$,  $S^\vee\cong H_3\dual$, and \eqref{eq:tdVntXhRtskHg} that $T_0:=\set{a_1,a_2,a_3,u,v,b,c}$ is a sublattice of $L$, and that $T_0$ is isomorphic to the lattice on the left in Figure \ref{fignaspolya}.
Using Lemma \ref{lemma:ggrprsprd} and \eqref{eq:knRlhKrhkn}, it is straightforward to see that $\Jir{\Con{T_0}}$ is correctly drawn in Figure \ref{fignaspolya}, $|\Con {T_0}|=5$, and $\cd{T_0}=5/64$.
Clearly, $\set{b,v}=\Jr{T_0}\subseteq \Jr L$ and 
$\set{c,u}=\Mr{T_0}\subseteq \Mr L$. 
These inclusions (which will also be used in the next sentence) and Lemma \ref{lemma:nofourjred} yield that $|\Jr L|$ and $|\Mr L|$ belong to $\set{2,3}$. 
If $|\Jr L|=|\Mr L|=2$, then $\Jr L\cup \Mr L\subseteq T_0$, whence Lemma \ref{lemma:pdgbpLsT}\eqref{lemma:apdgbpLsT} together with $\cd{T_0}=5/64$ yields that $\cd L\leq \cd{T_0}=5/64$, a contradiction.

Therefore, we may assume that one of 
$|\Jr L|$ and $|\Mr L|$ equals $3$. In fact, duality allows us to assume that $|\Jr L|=3$. Let $T:=T_0\cup \Jr L\cup \Mr L$. Pick an element $z\in T$ such that $c\prec_T z\leq a_2$, and observe that 
\begin{equation}
z\parallel a_1\text{ and }z\parallel a_3.
\label{eq:trRnkszBrfJ}
\end{equation}
Indeed,  since $z\leq a_1$ would lead to $z\leq a_1\wedge a_2=u$, $a_1\leq z$ to $a_1\leq a_2$, 
$z\leq a_3$ to $z\leq a_2\wedge a_3=c$,  and $a_3\leq z$ to $a_3\leq a_2$, we infer the validity of  \eqref{eq:trRnkszBrfJ}.

By Lemma \ref{lemma:pdgbpLsT}\eqref{lemma:apdgbpLsT},  $T$ is a sublattice of $L$ and $\cd T\leq \cd L$. 
Furthermore, $|\Jr L|=3$, $|\Mr L|\leq 3$, $\set{b,v}\subseteq \Jr L\cap T_0$, and $\set{c,u}\subseteq \Mr L\cap T_0$ imply that $|T\setminus T_0|\leq 2$. This inequality yields that at least one of the edges $(u,a_1)$, $(u,c)$, $(c,a_3)\in\Edge{T_0}$ remains an edge of $T$. 
If $(u,a_1)\in\Edge T$, then (using that $T_0$ is a sublattice of $T$, 
while considering the intervals in $T$)
$[u,a_1]\pup[c,b]\supseteq [c,z]$ gives $\con_T(c,z)\leq \con_T(u,a_1)$, whence Lemma \ref{lemma:khVnMtTlkBr} together with  \eqref{eq:trRnkszBrfJ} implies $\cd L\leq \cd T\leq 6/64$, a contradiction. 
If $(u,c)\in \Edge T$, then $[u,c]\pup [a_1,b]\pdn [u,a_2]\supseteq [c,z]$ implies $\con_T(c,z)\leq \con_T(u,c)$, and Lemma \ref{lemma:khVnMtTlkBr} gives the same contradiction as above.  
Finally, if $(c,a_3)\in \Edge T$, then  
\begin{equation*}
[c,a_3]\pup [a_2,v]\supseteq[a_2,b]\pdn[u,a_1]\pup[c,b]\supseteq [c,z]
\end{equation*}
implies $\con_T(c,z)\leq\con_T(c,a_3)$.  Thus Lemma \ref{lemma:khVnMtTlkBr} together with  \eqref{eq:trRnkszBrfJ}  gives  $\cd L\leq\cd T\leq 6/64$, a contradiction again. Consequently, we have excluded the first possibility, $a_2\wedge a_3>u$, in \eqref{eq:fFjmlVgysBck}.

Therefore, $a_1\wedge a_2>u$, which is the second possibility in \eqref{eq:fFjmlVgysBck}. 
Forgetting its meaning in the previous paragraph, let $c$ stand for $a_1\wedge a_2$. 
Reusing the symbol $S$, we claim that
\begin{equation}
S :=\set{a_1,a_2,a_3,b,c,u,v}\text{ is a sublattice of }L\text{ and that }S\cong N_{5,5}.
\label{eq:vmknVbdMxr}
\end{equation}
We already know $S^\wedge \cong H_3$ with $a_1\wedge a_2=c>u$ and $S^\vee \cong H_3\dual$ with $a_1\vee a_2=b<v$. 
Thus, by duality, to prove \eqref{eq:vmknVbdMxr}, it suffices to show that $c\vee a_3=v$.
Suppose the contrary, and denote $c\vee a_3$ by $x$.
Then $a_3<x<v$; see the right part of Figure \ref{fignaspolya}. 
At this point, $a_1$ and $a_2$ play symmetric roles. 
For $i\in\set{1,2}$,  the inequality $x\leq a_i$ would imply $a_3\leq a_i$, while $a_i\leq x$ would yield 
$v=a_i\vee a_3\leq x$; that is, we would get contradictions. 
Hence, $x\parallel a_1$ and $x\parallel a_2$, whence $\set{a_1,a_2,x}$ is a three-element antichain. 
Since Cases 1, 2, and 3 have been excluded for every 3-element antichain, the meet-subsemilattice generated by $\set{a_1,a_2,x}$ is isomorphic to $H_3$. 
Therefore, using $a_1\wedge a_2=c\leq x$, exactly one of the meets $a_1\wedge x$ and $a_2\wedge x$ equals $c$. 
By symmetry, we may assume that $a_1\wedge x=c$ and $y:=a_2\wedge x>c$; see  the right of Figure \ref{fignaspolya}. 
Similarly to its use when proving \eqref{eq:Scorrdep}, 
Lemma \ref{lemma:pdgbpLsT}\eqref{lemma:bpdgbpLsT} implies that $S^+:=S\cup \{x,y\}$ is a sublattice of $L$, that it is correctly depicted, and that $\cd L\leq \cd {S^+}$. Using $a_3\nleq a_2$ or the figure, we obtain $a_3\nleq y$. 
Since  $[u,a_3]\pup [c,x]\supseteq [c,y]$ holds in $S^+$, we have $\con_{S^+}(c,y)=:\beta\leq\alpha:=\con_{S^+}(u,a_3)$. Hence Lemma \ref{lemma:khVnMtTlkBr} implies $\cd {S^+}\leq 6/64$, whence $\cd L\leq \cd {S^+}\leq 6/64$ is a contradiction. This proves \eqref{eq:vmknVbdMxr}.

Finally, since $S$ is generated by the antichain $\set{a_1,a_2,a_3}$, 
\eqref{eq:vmknVbdMxr} completes the proof of 
Lemma \ref{lemma-3a-cha}.
\end{proof}

\begin{lemma}\label{lemma:sJqsM} If $\cd L>6/64$ holds for a finite lattice $L$, then $|\Jr L|=|\Mr L|$. 
\end{lemma}

\begin{proof} 
First, assume that $L$ has an element $v$ with more than two lower covers or more than two covers (i.e., upper covers).
By duality, we may assume that $a_1$, $a_2$, and $a_3$ are three distinct lower covers of $v$. Since the sublattice generated by the antichain $\set{a_1,a_2,a_3}$ is clearly not isomorphic to $N_{5,5}$, 
the conclusion of  Lemma \ref{lemma-3a-cha} fails here, whence its hypothesis cannot hold. Therefore $\Core L\cong M_3$, and $|\Jr L|=1=|\Mr L|$, as required.

Second, assume that each element has at most two lower covers and at most two covers.   We count the number $|\Edge L|$ of edges according to their top elements: 
$0$ is the top of no edge; 
each of the $|\Jr L|$ join-reducible elements is the top of two edges; 
and each of the $|\Jir L|=|L|-1-|\Jr L|$ join-irreducible elements is the top of a single edge. Thus 
$|\Edge L|=2\cdot |\Jr L| + (|L|-1-|\Jr L|)=|L|-1+|\Jr L|$. By duality, $|\Edge L|=|L|-1+|\Mr L|$. Comparing the two expressions for $|\Edge L|$,  we obtain the required equality $|\Jr L|=|\Mr L|$, completing the proof.
\end{proof}

\begin{lemma}\label{lemma:nofour}
If $L$ is a finite lattice with $\cd L>6/64$, then $L$ has no four-element antichain.
\end{lemma}

\begin{proof} 
Suppose the contrary, and let $\set{a_1,a_2,a_3,a_4}$ be a four-element antichain in $L$. Clearly, $\Core L\ncong M_3$.  Hence, by Lemma \ref{lemma-3a-cha}, for every three-element  $X\subseteq L$,
\begin{equation}
\text{if }X\subseteq \set{a_1,\dots,a_4},\text{ then }X\text{ generates a sublattice isomorphic to }N_{5,5}.
\label{eq:brTTvRfhnZk}
\end{equation}
Let $B:=\set{a_i\vee a_j: 1\leq i<j\leq 4}$.  
Lemma \ref{lemma:nofourjred}, combined with $B\subseteq \Jr L$ and \eqref{eq:brTTvRfhnZk}, implies that $2\leq |B|\leq 3$. Moreover, \eqref{eq:brTTvRfhnZk} shows that each element of $B$ is comparable to at least one other element of $B$. Hence there are only four possibilities for the poset $(B;\leq)$. These will be discussed in  Cases $(\alpha)$, \dots, $(\delta)$.

\logicbreak
Case ($\alpha$): We assume that $B$ is a two-element chain. Since \eqref{eq:brTTvRfhnZk} holds for $\set{a_1,a_2,a_3}$, we may assume by symmetry that  $B=\set{v_1,v_2}$ with $v_1:=a_1\vee a_2<a_1\vee a_3 =a_2\vee a_3 =:v_2$.  Since \eqref{eq:brTTvRfhnZk} holds for  $\set{a_1,a_2,a_4}$, $\set{a_1\vee a_4, a_2\vee a_4} \subseteq \set{v_1,v_2}$, and $a_1\vee a_2=v_1<v_2$, we have that $a_1\vee a_4=a_2\vee a_4=v_2$. Hence,  
using $a_3\vee a_4\in\set{v_1,v_2}$, $v_1<v_2=a_2\vee a_3=a_2\vee a_4$, and applying  \eqref{eq:brTTvRfhnZk} to  $\set{a_2,a_3,a_4}$, we obtain $a_3\vee a_4=v_1$. Therefore, 
$v_2=a_2\vee a_3 \leq (a_1\vee a_2)\vee (a_3\vee a_4)=v_1\vee v_1=v_1$ is a contradiction, as required.


\begin{figure}[ht] 
\centerline{ \includegraphics[width=\figwidthcoeff\textwidth]{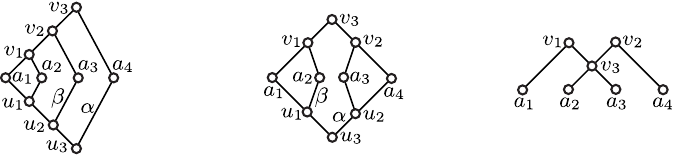}} 
\caption{Illustrations for the proof of Lemma \ref{lemma:nofour}}\label{figszeder}
\end{figure}

\logicbreak
Case ($\beta$): We assume that $B$ is a three-element chain.  Let $B=\set{v_1,v_2,v_3}$ with $v_1<v_2<v_3$. First, we show that there exists a three-element $Y\subseteq\set{a_1,\dots,a_4}$ that generates a sublattice containing both $v_1$ and $v_2$. Suppose the contrary. Then, by symmetry, we may assume that $a_1\vee a_2=v_1$.  Applying  \eqref{eq:brTTvRfhnZk} to $\set{a_1,a_2,a_3}$ and $\set{a_1,a_2,a_4}$, the nonexistence of such a $Y$ implies $a_1\vee a_3=v_3$,   $a_2\vee a_3=v_3$, 
$a_1\vee a_4=v_3$,  and  $a_2\vee a_4=v_3$. Since five of the six possible joins are different from $v_2$,  the remaining join equals $v_2$, that is, $a_3\vee a_4=v_2$. Hence
\begin{equation*}
v_3 = a_1\vee a_3 \leq (a_1\vee a_2)\vee (a_3\vee a_4)=v_1\vee v_2=v_2,
\end{equation*}
which is a contradiction. This shows that $Y$ exists. 

By symmetry again, we may assume that $Y=\set{a_1,a_2,a_3}$ with $a_1\vee a_2=v_1$. By \eqref{eq:brTTvRfhnZk} applied to $Y$, we have $a_1\vee a_3=v_2$ and $a_2\vee a_3=v_2$.
To obtain $v_3$, we need $a_4$. That is, $a_s\vee a_4=v_3$ for some $s\in\set{1,2,3}$. 
If we had $a_2\vee a_4\neq v_3$ and $a_3\vee a_4\neq v_3$, then $s$ would be $1$ and \eqref{eq:brTTvRfhnZk} would fail for $\set{a_1,a_2,a_4}$. 
Thus $a_2\vee a_4=v_3$ or $a_3\vee a_4=v_3$. 
If $a_2\vee a_4=v_3$, then $a_2\vee a_3=v_2<v_3$ together with 
 \eqref{eq:brTTvRfhnZk} applied to $\set{a_2,a_3,a_4}$
yields $a_3\vee a_4=v_3$. 
Similarly, 
if $a_3\vee a_4=v_3$, then $a_2\vee a_3=v_2<v_3$, whence 
 \eqref{eq:brTTvRfhnZk} applied to $\set{a_2,a_3,a_4}$
yields $a_2\vee a_4=v_3$. Therefore, $a_2\vee a_4=v_3$ and $a_3\vee a_4=v_3$. 
We have established five equations: $a_1\vee a_2=v_1$, $a_1\vee a_3=v_2$, $a_2\vee a_3=v_2$, $a_2\vee a_4=v_3$, and $a_3\vee a_4=v_3$.  
The sixth one, $a_1\vee a_4=v_3$, follows by applying  \eqref{eq:brTTvRfhnZk} to $\set{a_1,a_2,a_4}$. 

Hence, the joins of the form $a_i\vee a_j$ are correctly visualized on the left side of Figure \ref{figszeder}. In other words, for $i\neq j$, $a_i\vee a_j=v_{\max{\set{i,j}}-1}$. The meets are also correctly visualized, since we know from \eqref{eq:brTTvRfhnZk} that the meets behave dually to the joins. For $i\in\set{1,2,3}$, let $u_i$ be defined by the rule that $a_j\wedge a_t=u_i$ if and only if $a_j\vee a_t=v_i$; see the figure. 
It follows from \eqref{eq:brTTvRfhnZk} (or, alternatively, from applying Lemma \ref{lemma:pdgbpLsT}\eqref{lemma:bpdgbpLsT} in a similar way to its use for \eqref{eq:Scorrdep}) that $S:=\{a_1,\dots,a_4,v_1,v_2,v_3,u_1,u_2,u_3\}$ is a sublattice of $L$, that  $\cd L\leq \cd S$, and that $S$ is correctly depicted on the left of Figure \ref{figszeder}.
Since $[u_3,a_4]\pup[u_2,v_3]\supseteq [u_2,a_3]$ implies $\con_S(u_2,a_3)=:\beta\leq \alpha:=\con_S(u_3,a_4)$, Lemma \ref{lemma:khVnMtTlkBr} applied to $S$ gives that $\cd L\leq \cd S\leq 6/64$. This is a contradiction, as required.

\logicbreak
Case ($\gamma$): We assume that $B=\set{v_1,v_2,v_3}$ with $v_1<v_3$, $v_2<v_3$, and $v_1\parallel v_2$.
By \eqref{eq:brTTvRfhnZk}, 
there is no $\set{i,j,k}\subseteq\set{1,\dots,4}$ such that $v_1=a_i\vee a_j$ and $v_2=a_j\vee a_k$. Hence, we may assume that $v_1=a_1\vee a_2$ and $v_2=a_3\vee a_4$. Applying \eqref{eq:brTTvRfhnZk}, it straightforwardly follows that 
both the join-subsemilattice and the meet-subsemilattice generated by $\set{a_1,\dots, a_4}$ are correctly depicted in the middle of Figure \ref{figszeder}.
We claim that $u_i\nleq v_{3-i}$ for $i\in\set{1,2}$. By symmetry, it suffices to consider $i=1$. 
If we had $u_1\leq v_2$, then---since $a_4\leq v_2$ and \eqref{eq:brTTvRfhnZk} applied to $\set{a_1,a_2,a_4}$ yields $u_1\vee a_4=v_3$---we would obtain $v_3= u_1\vee a_4\leq u_1\vee v_2=v_2$, which is impossible.
Hence  $u_i\nleq v_{3-i}$ for $i\in\set{1,2}$, indeed. 
This allows us to apply Lemma \ref{lemma:pdgbpLsT}\eqref{lemma:bpdgbpLsT} in a way similar to its use for \eqref{eq:Scorrdep} to conclude that $S:=\{a_1,\dots,a_4,u_1,u_2,u_3,v_1,v_2,v_3\}$ is a sublattice of $L$, that $S$ is correctly drawn as a \emph{lattice} in the middle of Figure \ref{figszeder}, and that $\cd L\leq \cd S$. 
Since $[u_3,u_2]\pup [u_1,v_3]\supseteq [u_1,a_2]$, we have that $\con_S(u_3,u_2)=:\alpha\geq \beta:=\con_S(u_1,a_2)$. 
Thus, Lemma \ref{lemma:khVnMtTlkBr} implies $\cd L\leq \cd S\leq 6/64$. This inequality is a contradiction, as required.

\logicbreak
Case ($\delta$): We assume that $B=\set{v_1,v_2,v_3}$ with $v_3<v_1$, $v_3<v_2$, and $v_1\parallel v_2$.
Apart from indexing, the situation is drawn on the right in Figure \ref{figszeder}. Since $v_1$, $v_2$, $v_3$, and $v_1\vee v_2$ are four distinct elements of $\Jr L$, Lemma \ref{lemma:nofourjred} implies $\cd L\leq 4/64$. Again, we have obtained a contradiction.

Since each of the four cases leads to a contradiction, the proof of Lemma \ref{lemma:nofour} is complete.  
\end{proof}

\begin{lemma}\label{lemma-3a-cHnb} If a finite lattice $L$ has a three-element antichain and $\cd L>6/64$, then
\begin{itemize}
\item either $\Core L\cong M_3$ and $\cd L=8/64$  (see \ceqref{th2} in Theorem \ref{thm:main}), 
\item or $\Core L\cong N_{5,5}$ and $\cd L=7/64$ (see \ceqref{th4} in Theorem \ref{thm:main}).
\end{itemize}
\end{lemma}

\begin{proof} 
Let $L$ be an $n$-element finite lattice with $\cd L>6/64$ and a three-element antichain $\set{a_1,a_2,a_3}$.
Since $\Core{M_3}=2/2^{5-1}=8/64$ and we know from \eqref{eq:cdLcdCorL} that $\cd L=\cd{\Core L}$, we may assume that $\Core L\ncong M_3$.

By Lemma \ref{lemma-3a-cha}, 
the sublattice $S_0$ generated by $\set{a_1,a_2,a_3}$ is isomorphic to $N_{5,5}$. Hence,  $|\Jr L|\geq 2$ and $|\Mr L|\geq 2$.  Apart from indexing, we have $b:=a_1\vee a_2<v:=a_1\vee a_3=a_2\vee a_3$ and $c:=a_1\wedge a_2>u:=a_1\wedge a_3=a_2\wedge a_3$; see the left part of Figure \ref{figdio}.
We claim that
\begin{equation}
|\Jr L|=|\Mr L| = 2.
\label{eq:mfMpfmGzPbk}
\end{equation}
Seeking a contradiction, suppose that \eqref{eq:mfMpfmGzPbk} fails.  
We know from Lemma \ref{lemma:sJqsM} that $|\Jr L|=|\Mr L|$. Hence Lemma \ref{lemma:nofourjred} and the failure of \eqref{eq:mfMpfmGzPbk} imply $|\Jr L|=|\Mr L| = 3$. 
Therefore, since $\set{u,c}=\Mr {S_0}\subseteq \Mr L$ and 
 $\set{b,v}=\Jr {S_0}\subseteq \Jr L$, there are unique elements $p$ and $q$ (not necessarily distinct and not necessarily outside $S_0$) such that $\Mr L=\set{u,c,p}$ and $\Jr L=\set{b,v,q}$.


\begin{figure}[ht] 
\centerline{ \includegraphics[width=\figwidthcoeff\textwidth]{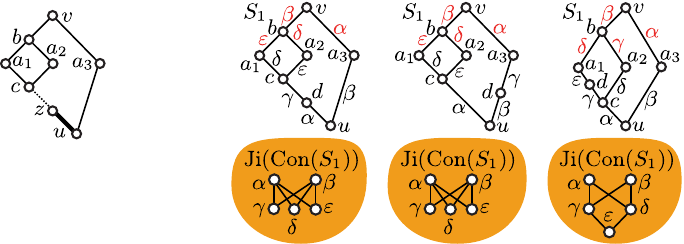}} 
\caption{A part of $S$ on the left and three possibilities for $S_1$
}\label{figbarack}
\end{figure}

Let $S:=S_0\cup \set{p,q}$. Since $\Mr L\cup \Jr L\subseteq S$, Lemma \ref{lemma:pdgbpLsT}\eqref{lemma:apdgbpLsT} yields that $S$ is a sublattice of $L$ and $\cd L\leq \cd S$. 
Since $S$ has been obtained from $S_0$ by adding only at most two new elements, at least two of the four prime intervals (that is, two-element intervals) $[u,c]_{S_0}$, $[u,a_3]_{S_0}$, $[b,v]_{S_0}$, and $[a_3,v]_{S_0}$ remain prime intervals of $S$. Thus, duality allows us to assume that $a_3\prec_S v$ or $b\prec_S v$ in $S$.
Before inspecting these two cases, pick an element $z\in S$ such that $u\prec_S z\leq c$, and take an edge $(x,y)\in\Edge S$ with $c\leq x\prec_S y\leq b$.  
The situation is shown on the left of Figure \ref{figbarack}, but note that the elements $p$, $q$, $x$ and $y$ are not specified and need not be drawn.
We are going to compute in $S_0\cong N_{5,5}$, but the intervals are understood in $S$.

First, assume $b\prec_S v$. Then $[b,v]\pdn [u,a_3]\pup [c,v]\supseteq [c,b]\supseteq[x,y]$ implies  $\con_S(x,y)\leq \con_S(b,v)$. Thus, letting $(x,y,b,v,S)$ play the role of $(u_1,v_1,u_2,v_2,L)$, we obtain from Lemma \ref{lemma:khVnMtTlkBr} that  
$\cd L\leq \cd S\leq 6/64$.

Second, assume $a_3\prec_S v$. With $(c,a_3,z)$ playing the role of $(a,b,x)$, Lemma \ref{lemma:oppedges} gives $\con_S(u,z)=\con_S(a_3,v)$. Thus 
$[a_3,v]\pdn [u,b]\supseteq[x,y]$ implies  $\con_S(x,y)\leq \con_S(a_3,v)=\con_S(u,z)$. Hence, letting the quadruple $(u,z,x,y,S)$ play the role of $(u_1,v_1,u_2,v_2,L)$, we obtain from Lemma \ref{lemma:khVnMtTlkBr} that $\cd L\leq \cd S\leq 6/64$. 

We have seen that regardless of whether  $a_3\prec_S v$ or $b\prec_S v$, we obtain $\cd L\leq  6/64$,  a contradiction that proves \eqref{eq:mfMpfmGzPbk}.

It follows from  \eqref{eq:mfMpfmGzPbk}, $\Mr{S_0}=\set{c,u}\subseteq\Mr L$, $\Jr{S_0}=\set{b,v}\subseteq\Jr L$, and Lemma \ref{lemma:pdgbpLsT}\eqref{lemma:apdgbpLsT} that for every subset $X$ of $L$,
\begin{equation}
\text{if }S_0\subseteq X,\text{ then }X\text{ is a sublattice of }L\text{ and }\cd L\leq \cd X.
\label{eq:rnKnmrtKrcdKVf}
\end{equation}

Next, we claim that $[u,v]$ is the union of the eight intervals corresponding to the eight edges of $N_{5,5}$ drawn on the left of Figure \ref{figdio}, that is,
\begin{equation}
[u,v]=[u,c]\cup[c,a_1]\cup [c,a_2]\cup [a_1,b]\cup[a_2,b] \cup [b,v]\cup [u,a_3]\cup [a_3,v].
\label{eq:msmrSnkpkcp}
\end{equation}
To derive a contradiction, suppose that $x\in [u,v]$ is outside the union above. Then $x\parallel a_3$. Since $\set{a_1,a_2,a_3,x}$ is not an antichain by Lemma \ref{lemma:nofour}, $x$ is comparable with $a_1$ or $a_2$. 
If $x$ were comparable with both $a_1$ and $a_2$, then $a_1\parallel a_2$ would imply that either $x\geq a_1$ and $x\geq a_2$, whence $x\in[b,v]$, or $x\leq a_1$ and $x\leq a_2$, whence $x\in [u,c]$. Thus, $x$ is comparable with exactly one of $a_1$ and $a_2$. 
By symmetry and duality, we may assume that $a_1\parallel x$ and $a_2<x$. Since $x\leq b$ would give $x\in[a_2,b]$ and $b\leq x$ would give $a_1\leq x$, the set $\set{b,x,a_3}$ is a three-element antichain.
Since $\Jr L=\set{b,v}$ by \eqref{eq:mfMpfmGzPbk}, we obtain that, in addition to $b\vee a_3=v$,  the equalities $b\vee x=v$ and $x\vee a_3=v$ also hold. 
Hence the antichain $\set{b,x,a_3}$ contradicts Lemma \ref{lemma-3a-cha}, and we have proved 
\eqref{eq:msmrSnkpkcp}.

Assume that there is an element $x\in\Nar L\cap[u,v]$. Since $a_1\parallel a_3$, either $x\leq a_1$ and $x\leq a_3$, whence $x=u$,
or $x\geq a_1$ and $x\geq a_3$, whence $x=v$. Hence $\Nar L\cap[u,v]\subseteq\set{u,v}$, and $[u,v]$ is glued sum indecomposable.
If there were an element $y\in L$ such that $y\parallel u$, then $y\wedge u$ would be a third meet-reducible element of $L$, contradicting $c,u\in\Mr L$ and \eqref{eq:mfMpfmGzPbk}. Hence, $u\in\Nar L$. Dually, $v\in \Nar L$, whence $L=\idl u\gsum [u,v]\gsum \fil v$. Furthermore, $\idl u$ and $\fil v$ are chains by  \eqref{eq:mfMpfmGzPbk}. Thus, $\Core L=[u,v]$.

Next, we claim that $[u,v]=S_0$. Seeking a contradiction, suppose the contrary.
Then $S_0\subset [u,v]$, whence we can pick an element $d\in [u,v]\setminus S_0$. Let $S_1:=S_0\cup\set{d}$; it is a sublattice and $\cd L\leq\cd{S_1}$ by \eqref{eq:rnKnmrtKrcdKVf}.
Using duality, the symmetry between $a_1$ and $a_2$, and \eqref{eq:msmrSnkpkcp}, we may assume that $u<d<c$, $u<d<a_3$, or $c<d<a_1$.  A trivial argument, similar to the one proving \eqref{eq:Scorrdep}, shows that these possibilities are correctly diagrammed in Figure \ref{figbarack}.
Using Lemma \ref{lemma:ggrprsprd} and \eqref{eq:knRlhKrhkn}, it is straightforward to verify that, in each of the three possibilities, the labeling is correct, the poset $\Jir{\Con{S_1}}$ is correctly drawn, and $|\Con{S_1}|$ is $11$ or $8$. 
Thus $\cd{S_1}$ is $5.5/64$ or $4/64$, whence $\cd{S_1}<6/64$. Hence 
$\cd L<6/64$, a contradiction. This shows that $\Core L=[u,v]=S_0\cong N_{5,5}$, as required. 

Finally, based on Lemma \ref{lemma:ggrprsprd}, the poset $\Jir{\Con{N_{5,5}}}$ is shown in Figure \ref{figdio}. 
By \eqref{eq:knRlhKrhkn}, $|\Con L|=7$. Therefore $\cd L=7/64$, and \ceqref{th4}  in Theorem \ref{thm:main} holds.
This completes the proof of Lemma \ref{lemma-3a-cHnb}.
\end{proof}

\section{If the width of $L$ is at most 2}
The \emph{width} of a finite lattice $L$, denoted by $\width L$, is the maximum size of an antichain in $L$. This section deals with finite core lattices $L$ such that $\width L=2$ and $\cd L>6/64$. There are several ways to observe that such a lattice is necessarily \emph{planar} (that is, it has a planar diagram); for example, this follows from \eqref{eq:233planar} or Lemma 2.2 of Cz\'edli and Schmidt \cite{czgschtJH}. In this section, we always work with a \emph{fixed planar diagram} of $L$, and terms like left or right refer to this diagram. Although our treatment is visual at some points, the relevant theory of planar lattices is rigorously presented in Kelly and Rival \cite{kellyrival}. (Neither here nor later in the paper does the reader need to consult \cite{kellyrival}.) Let  $\Bndl L$ and  $\Bndr L$ denote the \emph{left boundary chain} and the \emph{right boundary chain} of (the fixed diagram of) $L$.

\begin{lemma}\label{lemma:blrspd}
For every finite lattice $L$  with $\width L=2$, we have that $L=\Bndl L\cup \Bndr L$.
\end{lemma}

\begin{proof} Suppose, for the sake of contradiction, that there exists an element $x\in L$ such that $x\notin \Bndl L\cup \Bndr L$. Since $\Bndl L$ is a maximal chain and $x\notin \Bndl L$, there is an element $y\in\Bndl L$ such that $y\parallel x$. Moreover, since $\Bndl L$ is the \emph{left} boundary chain, $y$ is to the left of $x$ (see the definition of the left boundary and Propositions 1.6 and 1.7 in Kelly and Rival \cite{kellyrival}). By left-right duality, we obtain an element $z\in\Bndr L$ such that $x$ is to the left of $z$ and, in particular, $x\parallel z$.  By a result of Zilber, appearing as Proposition 1.7 in Kelly and Rival \cite{kellyrival}, the relation \quot{to the left of}  is transitive. Hence, $y$ is to the left of $z$ and, in particular, $y\parallel z$. Thus, $\set{y,x,z}$ is a 3-element antichain. This antichain violates $\width L=2$ and completes the proof of Lemma \ref{lemma:blrspd}.
\end{proof}

\begin{lemma}\label{lemma:w2}
Let $L$ be a finite lattice such that $\width L=2$ and $6/64<\cd L\leq 8/64$. 
Then one  of the possibilities 
 \ceqref{th3},  \ceqref{th5}, \ceqref{th6}, 
\ceqref{th7}, and \ceqref{th8} listed in  Theorem \ref{thm:main} holds.
\end{lemma}

\begin{proof} Since $\cd L=\cd{\Core L}$  by \eqref{eq:cdLcdCorL} and $\width L=\width{\Core L}$, we may assume that $L$ is a core lattice. 
Furthermore, $|L|>1$, since otherwise $\cd L$ would be $64/64$. 
Let $L=L_1\gsum\dots\gsum L_t$ be the canonical glued sum decomposition of $L$; here $t\in\Nplu$ and none of the $L_1$, \dots, $L_t$ is a chain. 
Since $L=\Core L$ is interesting only up to $\rgsiso$, we may assume that $\cd{L_1}\geq \dots \geq \cd{L_t}$. Cz\'edli \cite[Lemma 2]{czg864}, applied to  $L=L_1\gsum\dots\gsum L_t$, asserts that
\begin{equation}
\cd  L=\cd{L_1}\cdots\cd{L_t},\text{ where }
\cd{L_1}\geq \dots \geq \cd{L_t}.
\label{eq:lmMrlCsdj}
\end{equation}

We know from Cz\'edli \cite[Lemma 10]{czg864}, or even from Cz\'edli \cite{czg-lconl2}, that $\cd{L_i}\leq 1/2$ for $i\in\set{1,\dots,t}$. Hence $\cd L\leq 1/2^t$. 
Thus, since $6/64<\cd L$, we have that $t\leq 3$.

First, assume that $t=3$. Then  $\cd L\leq 1/2^3=8/64$. If $\cd{L_i}<1/2=32/64$ for some $i\in\set{1,2,3}$, 
then $\cd{L_i}\leq 20/64$ by Lemma \ref{lemma:earlylist}, whence 
\begin{equation*}6/64<\cd L= \cd{L_1}\cd{L_2}\cd{L_3} \leq 1/2\cdot1/2\cdot 20/64=5/64,
\end{equation*} 
which is a contradiction. Hence $\cd{L_i}=1/2$  and, by Lemma \ref{lemma:earlylist}, $L_i\cong B_4$ for all $i\in\set{1,2,3}$. Furthermore, $\cd L=1/2^3=8/64$ and $\Core L=L=B_4\gsum B_4\gsum B_4$. That is, \ceqref{th3} holds, as required.

Second, assume that $t=2$. Observe that $|\Nar{L_1}|=|\Nar{L_2}|=2$, since otherwise $t$ would be larger than $2$. 
As the first subcase, we assume $\cd{L_1}=1/2$. 
By Lemma \ref{lemma:earlylist}, $L_1\cong B_4$.
Since  $6/64< \cd L= \cd{L_1}\cd{L_2}\leq 8/64$, we have
$12/64< \cd{L_2}\leq 16/64$. Hence, taking $|\Nar{L_2}|=2$ into account as well, $L_2=\Core {L_2}$ satisfies
\eeqref{lem4} or \eeqref{lem6} of Lemma \ref{lemma:earlylist}. That is, 
$L_2$ is an edge gluing of two copies of $B_4$ and $\cd{L_2}=16/64$, or $L_2$ belongs to $\Circ 6$ and $\cd{L_2}=14/64$.  Thus, taking $L_1\cong B_4$,  $\cd{B_4}=1/2$, and \eqref{eq:lmMrlCsdj} into account as well, 
 \ceqref{th3} or  \ceqref{th5} holds for  $L=\Core L$, as required.

In the second subcase of the case $t=2$, we assume that $\cd{L_1}<1/2$. Then $\cd{L_1}=20/64$ or $\cd{L_1}\leq 16/64$
by Lemma \ref{lemma:earlylist}, and  $\cd{L_2}\leq \cd{L_1}$ by \eqref{eq:lmMrlCsdj}. 
We focus on $\cd{L_2}$. 
If we had $\cd{L_2}\leq 16/64$, then  
$6/64<\cd L=\cd{L_1}\cd{L_2}\leq 20/64\cdot 16/64=5/64$ would be a contradiction. Hence, $\cd{L_2}= \cd{L_1}=20/64$, and Lemma \ref{lemma:earlylist} implies $L_1\cong L_2\cong N_5$. Hence $L\cong N_5\gsum N_5$  and $\cd L=\cd{L_1}\cd{L_2}=20/64\cdot 20/64=6.25/64$. Therefore \ceqref{th8} of Theorem \ref{thm:main} holds, as required.


\begin{figure}[ht] 
\centerline{ \includegraphics[width=\figwidthcoeff\textwidth]{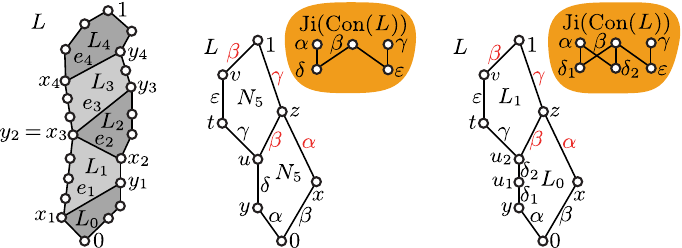}} 
\caption{Illustrations for the proof of Lemma \ref{lemma:w2} }\label{figszolo}
\end{figure}

Third,  we assume $t=1$. So in the rest of the proof, $L=\Core L$, $6/64<\cd L\leq 8/64$, and $|\Nar L|=2$.  The intersection $\Bndl L\cap\Bndr L$ of the two boundary chains is $\Nar L=\set{0,1}$. Furthermore, $\Bndl L\cup\Bndr L=L$ by Lemma \ref{lemma:blrspd}.
Let us call an edge $(x,y)\in\Edge L$ a \emph{chord edge} if $\set{x,y}\cap\set{0,1}=\emptyset$ and $x$ lies in one of the two boundary chains while $y$ lies in the other. For example, the chord edges of $L$ on the left of Figure \ref{figszolo} are $e_1=(x_1,y_1)$, \dots, $e_4=(x_4,y_4)$.
Since the diagram is planar, no two distinct chord edges intersect at their internal geometric points.

Let $h$ be the number of chord edges, and let $n:=|L|$.
If we had $h=0$, then $L$ would belong to $\Circ n$, whence  
\eeqref{lem2}, \eeqref{lem3}, 
\eeqref{lem6}, \eeqref{lem7}, and \eeqref{lem10} of Lemma \ref{lemma:earlylist} would imply $\cd L>8/64$, which would be a contradiction. Hence $h>0$.

For a chord edge $e=(x,y)$, $L$ is decomposed as the edge gluing of $\idl y$ and $\fil x$ (along the edge $e$). If $\idl y$ or $\fil x$ contains a chord edge of $L$ distinct from $e$, then this decomposition can be repeated. 
Therefore, there are integers $n_0, n_1,\dots, n_h$ greater than 3 and lattices $L_i\in\Circ{n_i}$ and $M_i$ for $i\in\set{0,\dots,h}$ such that $M_0=L_0$, $M_h=L$, and $M_j$ is an edge gluing of $M_{j-1}$ and $L_j$ for $j\in\set{1,\dots, h}$.
For example, for $L$ on the left of Figure \ref{figszolo}, we have $(n_0,n_1,n_2,n_3,n_4)=(5,7,5,7,6)$,
$L_0=M_0=\idl{y_1}$, $L_4=\fil{x_4}$, $L_i=[x_i,y_{i+1}]$ for $i\in\set{1,2,3}$, $M_4=L$, and $M_j=\idl{y_{j+1}}$ for $j\in\set{1,2,3}$.

By Lemma \ref{lemma:earlylist}, $\cd{M_0}=\cd{L_0}\leq 1/2$. Furthermore, 
Lemma \ref{lemma:circedgglu}\eqref{lemma:acircedgglu} yields that  $\cd{M_j}\leq \cd{M_{j-1}}/2$ for $j\in\set{1,\dots,h}$. Hence a trivial induction on $j$ gives  $\cd L=\cd{M_h}\leq 1/2^{h+1}$. Combining this inequality with $6/64<\cd L$, we conclude that $h<3$. So we have to deal only with $h=1$ and $h=2$.

Assume that $h=1$. We know from the appropriate parts of Lemma \ref{lemma:earlylist}, or more directly from 
Cz\'edli \cite[Lemma 7]{czg864}, that $\cd{L_0}=(8+3/2^{n_0-7})/64$. Thus, the assumption $\cd L>6/64$ and
Lemma \ref{lemma:circedgglu}\eqref{lemma:acircedgglu} yield that
\begin{equation*}
\frac 6{64}<\cd L \leq \cd{L_0}/2< \frac{8+3/2^{n_0-7}}{64} \cdot \frac 12,
\end{equation*}
which implies $n_0\leq 6$. We obtain $n_1\leq 6$  similarly. Hence, $n_0,n_1\in\set{4,5,6}$. 
We may assume that $n_0\geq n_1$, since otherwise we could work with the dual of $L$.

If $(n_0,n_1)=(6,4)$, then $\cd L=7/64$ by Lemma \ref{lemma:earlylist}\eeqref{lem6} and Lemma \ref{lemma:circedgglu}\eqref{lemma:bcircedgglu}, whence  \ceqref{th6} of Theorem \ref{thm:main} holds. 
If $(n_0,n_1)=(5,5)$, then 
a straightforward computation, which is illustrated in the middle of Figure \ref{figszolo},  yields that $\cd L=6.5/64$, and so \ceqref{th7} is fulfilled; note that the details of a similar (but slightly longer) computation will soon be given for the right part of  Figure \ref{figszolo}, and also that $L$ need not be isomorphic to the lattice depicted in the middle of the figure. 
 
Since $\cd L\leq 8/64$, \eeqref{lem4} and \eeqref{lem8} of Lemma \ref{lemma:earlylist}
exclude the possibility that $(n_0,n_1)\in \{(4,4), (5,4)\}$. Excluding the remaining possibility that  $(n_0,n_1)\in \set{(5,6), (6,6)}$ requires more work.

The right part of Figure \ref{figszolo} visualizes how to rule out the case  $(n_0,n_1)=(6,5)$.
Although $L_0$ and $L$ need not be isomorphic to the lattices depicted in the figure, we can always choose the notation as follows. In $L_1$, let $z$ denote the unique atom such that $\jipair z$ (defined in \eqref{eq:pkmncShZLt}) is the gluing edge, and let $t$ be the other atom. In $L_0$, let $x$ be the unique atom such that $[\locov z,z]\pdn [\locov x,x]$, and let $y$ denote the other atom of $L_0$. Define $\alpha:=\con(\jipair y)$, $\beta:=\con(\jipair x)$, and $\gamma:=\con(\jipair t)$. 
Regardless of which of the two members of $\Circ 6$ is $L_0$  and how the gluing edge in $L_0$ and $L_1$ is situated, denote the two elements of $\Jir{L_0}\setminus\set{x,y}$ by $u_1$ and $u_2$. 
Set $\delta_i:=\con(\jipair{u_i})$ for $i\in\set{1,2}$. 
Let $v$ be the unique element of $\Jir{L_1}\setminus\set{z,t}$, and set $\epsilon:=\con(\jipair v)$. 
Since $\Jir L=\set{y,x,u_1,u_2,t,v}$, it follows from  \eqref{eq:mKflvGmgfll} that $\set{\alpha,\beta,\gamma,\delta_1,\delta_2,\epsilon}=\Jir{\Con L}$.
Using Lemma \ref{lemma:ggrprsprd}(A), it is straightforward to draw the poset $(\Jir{\Con L};\leq)$; see the figure. To count the ideals of this poset, note that $2^3=8$ ideals are disjoint from $\set{\alpha,\beta,\gamma}$; $2^2=4$ contain $\beta$;  
2 contain $\alpha$ but neither $\beta$ nor $\gamma$; 
$2^2=4$ contain $\gamma$ but neither $\alpha$ nor $\beta$; and exactly 1 contains $\alpha$ and $\gamma$ but not $\beta$. Hence $|\Idl{\Jir {\Con L};\leq}|=8+4+2+4+1=19$, and \eqref{eq:knRlhKrhkn} gives that $|\Con L|=19$. 
Therefore $\cd L=19/2^{|L|-1}=4.75/64$, which contradicts the assumption $6/64< \cd L$ and excludes the case  $(n_0,n_1)=(6,5)$.

Finally, if $(n_0,n_1)=(6,6)$, then $L$ 
has a sublattice $L'$ such that the parameters of $L' $ satisfy $(n_0,n_1)=(6,5)$, and $L$ is a dismantlable extension of $L'$; see \eqref{eq:dismnTlb}. We have seen that $\cd{L'}=4.75/64$, and we know from \eqref{eq:dsmntlExtsm}  that $\cd L\leq \cd{L'}$. Thus $\cd L\leq 4.75/64$, which contradicts the inequality $6/64< \cd L$ and excludes the possibility $(n_0,n_1)=(6,6)$.
This completes the case $h=1$ (within the case $t=1$).

Next, we assume that $h=2$. If $(n_0,n_1,n_2)=(4,4,4)$, then $\cd L=1/2^3=8/64$ by Lemma \ref{lemma:circedgglu}\eqref{lemma:bcircedgglu} and $L$ satisfies \ceqref{th3}, as required.

Thus it suffices to rule out all cases with $(n_0,n_1,n_2)\neq(4,4,4)$.  Let, say, $n_2\neq 4$. Then $n_2\geq 5$. Since $L_2\in\Circ{n_2}$ is a dismantlable extension of $N_5$, \eqref{eq:dsmntlExtsm} implies $\cd{L_2}\leq \cd{N_5}$. Combining this inequality with Lemma \ref{lemma:earlylist}\eeqref{lem3}, we have $\cd {L_2}\leq 20/64$.
Denote the sublattice $L_1\cup L_2$ in $L$ by $T$; then $T$ is an edge gluing of $L_1$ and $L_2$, while $L$ is an edge gluing of $L_0$ and $T$. Hence,
Lemma \ref{lemma:circedgglu}\eqref{lemma:acircedgglu} yields that $\cd T\leq \cd{L_2}/2\leq 10/64$ and 
$\cd L\leq \cd T/2\leq 5/64$,  which contradicts $6/64<\cd L$. 
We similarly obtain the same contradiction when $n_0$ or $n_1$ is distinct from 4, except that in this case 
we first obtain $M_1$ as an edge gluing of $L_0$ and $L_1$, and then $L$ is an edge gluing of $M_1$ and $L_2$. 
Therefore, the possibility that $(n_0,n_1,n_2)\neq(4,4,4)$ is ruled out, and the proof of Lemma \ref{lemma:w2} is complete.
\end{proof}

\section{Final deductions}
From the lemmas established earlier in the paper, we can now easily derive the results stated in Sections \ref{sect:intro} and \ref{sect:mainthm}.

\begin{proof}[Proof of Theorem \ref{thm:jm}] Let $L$ be a finite lattice with $\cd L>3/32$.
For $k\in\Nnul$, let $\ncu k$ and $\ncd k$ denote the number of elements of $L$ with exactly $k$ \ul upper \ul covers and the number of elements with exactly $k$ \ul lower \ul covers, respectively.
It follows from \eqref{eq:zsszmrdLjVntWz} and Lemma \ref{lemma:sJqsM} that  $\ncu 1=|\Jir L|=|\Mir L|=\ncd 1$.  The equality $\ncu 0=1=\ncd 0$ is trivial. Combining the already established equations with the fact, given by  Lemma \ref{lemma:nofour}, that no element has more than three upper covers or lower covers, we obtain that  
\begin{equation}
\ncu k=0=\ncd k\text{ for all }k\geq 4 \text{ and } \ncu k=\ncd k  \text{ for all }k\in\Nnul \setminus\set{2,3}.
\label{eq:hsvgDcszNjn}
\end{equation}
To handle the remaining values $k\in\set{2,3}$, we distinguish two cases.

First, assume that $\ncu 3=0=\ncd 3$. Then \eqref{eq:hsvgDcszNjn} together with  Lemma \ref{lemma:sJqsM} yields that $\ncu 2=|\Jr L| =|\Mr L|= \ncd 2$. Hence \jmref{jm2} holds.

Second, assume that  $\ncu 3>0$ or  $\ncd 3>0$. By duality, we may assume that $\ncd 3>0$. Then an interval $[u,v]$ of $L$ is isomorphic to $M_3$ by Lemma \ref{lemma:3lowcov}.  
Since the three lower covers of $v$ generate $[u,v]\cong M_3$, and $M_3$ is not isomorphic to $N_{5,5}$, the conclusion of  Lemma \ref{lemma-3a-cha} fails in this case, whence its hypothesis cannot hold. 
Hence $\Core L\cong M_3$. Thus $\ncu 3=1=\ncd 3$ and  $\ncu 2=0=\ncd 2$. Hence \jmref{jm2} holds again, and therefore it holds in both cases.

Finally, \jmref{jm1}, which asserts that $\ncu 1=\ncd 1$, follows from  \jmref{jm2}. The proof of Theorem \ref{thm:jm} is complete.
\end{proof}

\begin{proof}[Proof of Observation \ref{obs:R6}]
Based on \eqref{eq:knRlhKrhkn},  Lemma \ref{lemma:ggrprsprd}, and the middle part of Figure \ref{figalma}, it is straightforward to see that $|\Con{R_6}|=|\Idl{\Jir{\Con{R_6}}}|=3$. Hence $\cd{R_6}=3/32$. Clearly, 
$|\Jir{R_6}|=3\neq 4 =|\Mir{R_6}|$.
\end{proof}

\begin{proof}[Proof of Theorem \ref{thm:main}]
Let $L$ be a finite lattice with $\cd L>6/64$. We can assume  $\cd L\leq 8/64$, since otherwise \ceqref{th1} holds. Then $1<\width L<4$ by 
Lemmas \ref{lemma:earlylist} and \ref{lemma:nofour}, whence $\width L\in\set{2,3}$. 
Therefore Lemmas \ref{lemma-3a-cHnb} and \ref{lemma:w2}
imply Theorem \ref{thm:main}.
\end{proof}

\section{Conclusion and outlook}
We have proved that a finite lattice $L$ with congruence density $\cd L$ larger than $3/32$ has the same number of join-irreducible and meet-irreducible elements (that is, $|\Jir L|=|\Mir L|$), and pointed out that $\cd L=3/32$ does not imply  $|\Jir L|=|\Mir L|$.  
In fact, our analysis has yielded a slightly stronger conclusion.
Moreover, we have presented a structural description of finite lattices $L$ with $\cd L> 3/32$.
Equivalently, we have described finite lattices $L$ with more than $3\cdot 2^{|L|-6}$ congruences. 
We guess that, with a lot of work, the threshold $3/32$ in the description might be slightly improved in the future. However, improving the threshold to, say, $10^{-8}$ seems unlikely even with computer assistance; see also Footnote \ref{foot:imposs}.

To raise an open problem, whose difficulty is unclear at this moment, let $M_1$ be the one-element lattice, and let $M_k$ be the $(k+2)$-element modular lattice of length 2 for $2\leq k\in\Nplu$; see the right part of Figure \ref{figalma} for $k\leq 4$. 
Then $\width{M_k}=k$. Does every finite lattice $L$ satisfy $\cd L\leq \cd{M_{\width L}}$? It follows from Lemma \ref{lemma:earlylist} and Theorem \ref{thm:main} that the answer is positive for lattices $L$ with $\width L\leq 3$. 

There have been analogous investigations in which \quot{lattices} or \quot{congruences} have been replaced by other structures, including \quot{semilattices}, or by other compatible relations, such as the unary \quot{subuniverse} relation;  see, for example, 
Ahmed, Salih, and Hale \cite{delbrinsemilat}, 
  Ahmed and Horv\'ath \cite{delbrinKHE}, Cz\'edli \cite{czg-83Sub}, \cite{czg-semlatmancon}, and \cite{czgnear83}, Kwuida and Mure\c{s}an \cite{kuwidamuresan}, and the Abstract of Zaja, Haje, and Ahmed \cite{zaya-haje-delbrin}. 
After these antecedents, similar investigations are likely to appear in the future.

\end{document}